\documentclass[11pt,a4paper,reqno]{amsart}
\usepackage{amsmath}
\usepackage{amsfonts}
\usepackage{amssymb}
\usepackage{amscd}
\usepackage{url}
\usepackage[pdftex,bookmarks=true]{hyperref}

\newcommand\Vol{{\operatorname{Vol}}}
\newcommand\rank{{\operatorname{rank}}}

\newcommand\R{{\mathbf{R}}}

\renewcommand\P{{\mathbf{P}}}
\newcommand\E{{\mathbf{E}}}

\newcommand\Z{{\mathbf{Z}}}

\newcommand\col{{\mathbf{c}}}
\newcommand\row{{\mathbf{r}}}

\newcommand\Ba{{\mathbf a}}

\newcommand\Bf{{\mathbf f}}

\newcommand\Bu{{\mathbf u}}
\newcommand\Bv{{\mathbf v}}
\newcommand\Bw{{\mathbf w}}
\newcommand\Bx{{\mathbf x}}
\newcommand\By{{\mathbf y}}
\newcommand\Bz{{\mathbf z}}
\newcommand\ep{\epsilon}

\theoremstyle{plain}
  \newtheorem{theorem}[subsection]{Theorem}

  \newtheorem{proposition}[subsection]{Proposition}
  
  \newtheorem{lemma}[subsection]{Lemma}
  \newtheorem{corollary}[subsection]{Corollary}

  \newtheorem{condition}{Condition}

\theoremstyle{remark}
  \newtheorem{remark}[subsection]{Remark}

\theoremstyle{definition}
  \newtheorem{definition}[subsection]{Definition}

\begin{document}

\title[On the least singular value of random symmetric matrices]{ On the least singular value of random symmetric matrices}

\author{Hoi H. Nguyen}
\address{Department of Mathematics, University of Pennsylvania, 209 South 33rd Street, Philadelphia, PA 19104, USA}
\email{hoing@math.upenn.edu}
\subjclass[2000]{15A52, 15A63, 11B25}

\maketitle

\begin{abstract}
Let $F_n$ be an $n$ by $n$ symmetric matrix whose entries are bounded by $n^{\gamma}$ for some $\gamma>0$. Consider a randomly perturbed matrix $M_n=F_n+X_n$, where $X_n$ is a {\it random symmetric matrix} whose upper diagonal entries $x_{ij}, 1\le i\le j,$ are iid copies of a random variable $\xi$. Under a very general assumption on $\xi$, we show that for any $B>0$ there exists $A>0$ such that $\P(\sigma_n(M_n)\le n^{-A})\le n^{-B}$. The proof uses an inverse-type result concerning the concentration of quadratic forms, which is of interest of its own.
\end{abstract}

\section{introduction}
Let $F_n$ be an $n$ by $n$ matrix whose entries are bounded by $n^{O(1)}$. Consider a randomly perturbed matrix $M_n=F_n+X_n$, where $X_n$ is a {\it random matrix} whose entries are iid copies of a random variable. It has been shown, under a very general assumption on $\xi$, that the singular value of $M_n$ cannot be too small.

\begin{theorem}\cite[Theorem 2.1]{TVcir}\label{theorem:TV}
Assume that $M_n=F_n+X_n$, where the entries of $F_n$ are bounded by $n^{\gamma}$, and the entries of $X_n$ are iid copies of a random variable of zero mean and unit variance. Then for any $B>0$, there exists $A>0$ such that 

$$\P(\sigma_n(M_n) \le n^{-A})\le n^{-B}.$$
\end{theorem}

Here $\sigma_n(M_n)$ is the smallest singular value of $M_n$, defined as 

$$\sigma_n(M_n):= \inf_{\|x\|=1}\|M_nx\|.$$

The dependence among the parameters in Theorem \ref{theorem:TV} were made explicitly in \cite{TVcomp}. Under the stronger assumption that $\xi$ has sub-Gaussian distribution, Rudelson and Vershynin \cite{RV} obtained an almost best possible estimate on the tail bound of $\sigma_n(M_n)$. For more results regarding this random matrix ensemble we refer the reader to \cite{RV,TVcir,TVcomp}.

One important application of Theorem \ref{theorem:TV} is a polynomial bound for the condition number of $M_n$.

\begin{corollary}\cite[Corollary 2.10]{TVcir}\label{cor:TV}
With the same assumption as in Theorem \ref{theorem:TV}, for any $B>0$, there exists $A>0$ such that 

$$\P(\sigma_1(M_n)/\sigma_n(M_n) \le n^{A})\ge 1-n^{-B}.$$ 
\end{corollary}

The condition number $\kappa(M)=\sigma_1(M)/\sigma_n(M)$ of a matrix $M$ plays a crucial role in numerical
linear algebra. The above corollary implies that if one perturbs a fixed matrix $F$ of small spectral norm by a (very general) random matrix $X_n$, the condition
number of the resulting matrix will be relatively small with high probability. This fact has some nice applications in theoretical computer science. (See for instance \cite{S,ST} for further discussions on these applications).

Another popular model of random matrices is that of {\it random symmetric matrices}; this
is one of the simplest models that has non-trivial correlations between the matrix entries. A significant new difficulty in the study of the singularity of $X_n$ (or of $M_n$ in general) is that the symmetry ensures that $\det(X_n)$ is a quadratic function of each row, as opposed to the regular random ensembles in which $\det(X_n)$ is a linear function of each row.

Answering an old question of Weiss, a recent result of Costello, Tao and Vu \cite{CTV} shows that $X_n$ is almost non-singular with high probability.

\begin{theorem}\cite[Theorem 1.3]{CTV}
Assume that the upper diagonal entries $x_{ij}$ of $X_n$ are iid Bernoulli random variables. Then $X_n$ is non-singular with probability $1-n^{-1/8+o(1)}$.
\end{theorem}

In \cite{Ng}, the current author improved the bound to any polynomial order.

\begin{theorem}\cite[Theorem 1.2]{Ng}
Assume that the upper diagonal entries of $X_n$ are iid Bernoulli random variables. Then $X_n$ is non-singular with probability $1-O(n^{-B})$ for any $B>0$.
\end{theorem}

The goal of this note is to study the smallest singular value of randomly perturbed matrices $M_n$ under general assumptions on $\xi$. 

\begin{condition}[Anti-concentration]\label{condition:space} Assume that $\xi$ has zero mean, unit variance, and there exist positive constants $c_1<c_2$ and $c_3$ such that
$$\P(c_1\le |\xi-\xi'|\le c_2) \ge c_3,$$
where $\xi'$ is an independent copy of $\xi$.
\end{condition}

Note that the last condition is more general than the assumption of boundedness of $(2+\eta)$-moment, for some $\eta>0$. 

\begin{theorem}[Main theorem]\label{theorem:singularvalue} Assume that the upper diagonal entries $x_{ij}$ of $X_n$ are iid copies of a random variable $\xi$ satisfying Condition \ref{condition:space}. Assume also that the entries $f_{ij}$ of the symmetric matrix $F_n$ satisfy $|f_{ij}|\le n^\gamma$ for some $\gamma>0$. Then for any $B>0$, there exists $A>0$ such that 

$$\P(\sigma_n(M_n) \le n^{-A}) \le n^{-B}.$$
\end{theorem}

\begin{remark}
By technical reasons, we assume that $F_n$ is symmetric, although we believe that the above results hold without this assumption.
\end{remark}

Our result implies a polynomial bound for the condition number of $M_n$, assuming that the spectral norm of $X_n$ is bounded (which is the case for most of matrix ensembles).

\begin{corollary}
With the same assumptions as of Theorem \ref{theorem:singularvalue}, and assume that $\|X_n\|=n^{O(1)}$. Then for any $B>0$, there exists $A>0$ such that 
$$\P(\kappa(M_n) \ge n^{A})\le n^{-B}.$$ 
\end{corollary}

As another application, we show that the determinant of random symmetric matrices is concentrated around its mean with high probability.

\begin{corollary}\label{cor:concentration}
Assume that the upper diagonal entries $x_{ij}$ of $X_n$ are iid copies of a random variable $\xi$ of zero mean, unit variance, and there is a constant $C>0$ such that $\P(|\xi|\le C)=1$. Assume furthermore that the entries $f_{ij}$ of the symmetric matrix $F_n$ also satisfy $|f_{ij}|\le C$. Let $B>0$ be any constant. Then with probability $1-O(n^{-B})$,

$$|\det(M_n)|=\exp(-O(n^{1/3} \log n))\E(|\det(M_n)|),$$

where the implied constants depend on $B$ and $C$.
\end{corollary}

This corollary refines an important case of \cite[Theorem 34]{TVlocal} obtained by Tao and Vu by a different method, which in turn compensates previously known results on the concentration of the determinant of non-symmetric random matrices (cf. \cite{Ba,CV,FRZ,TVdet}).

{\bf Notation}. For a matrix $M$ we use the notations $\row_i(M)$ and $\col_j(M)$ to denote its $i$-th row vector and its $j$-th column vector respectively; we use the notation $(M)_{ij}$ to denote its $ij$ entry. 

As usual, we use $\eta$ to denote random Bernoulli variables (thus $\eta$ takes values $\pm 1$ with probability 1/2). For a given $0\le \mu\le 1$, we use $\eta^{(\mu)}$ to denote random Bernoulli variables of parameter $\mu$ (thus $\eta^{(\mu)}$ takes values $\pm 1$ with probability $\mu/2$ and zero with probability $1-\mu$). 

Here and later, asymptotic notations  such as $O, \Omega, \Theta$, and so for, are used under the assumption that $n \rightarrow \infty$. A notation such as $O_{C}(.)$ emphasizes that the hidden constant in $O$ depends on $C$. If $a= \Omega (b)$, we write $b \ll a$ or $a \gg b$. If $a=\Omega(b)$ and $b=\Omega(a)$, we write $a\asymp b$.

\section{The approach to prove Theorem \ref{theorem:singularvalue}}

For the sake of simplicity, we will prove our result under the following condition. 

\begin{condition}\label{condition:bound}
With probability one,

$$|x_{ij}| \le n^{B+1},$$

for all $i,j$.
\end{condition}

In fact, because $\xi$ has unit variance, we have 

$$\P(|x_{ij}|\ge n^{B+1})= O(n^{-2B-2}).$$ 

Thus, we can assume that $|x_{ij}|\le n^{B+1}$ at the cost of an additional negligible term $o(n^{-B})$ in probability.

We next assume that $\sigma_n(M_n)\le n^{-A}$. Thus

$$M_n\Bx=\By,$$ 

for some $\|\Bx\|=1$ and $\|\By\| \le n^{-A}$. There are two cases to consider.

{\bf Case 1.} $\det(M_n)=0$. In this case $\xi$ has discrete distribution.

We first show that it is enough to consider the case of $M_n$ having rank $n-1$, thanks to the following result.

\begin{lemma}\label{lemma:rankincrease}  For any $1\le k \le n-2$,
$$\P(\rank(M_n)=k \le n-2) \le \Omega_{c_1}(1)\P(\rank (M_{2n-k-1})=2n-k-2).$$
\end{lemma}

We deduce Lemma \ref{lemma:rankincrease} from a useful observation by Odlyzko, whose simple proof is presented in Appendix \ref{appendix:Odlyzko}.  

\begin{lemma}[Odlyzko's lemma,\cite{O}]\label{lemma:Odlyzko}
Let $H$ be a linear subspace in $\R^n$ of dimension at most $k\le n$. Then 
$$\P(\Bu\in H)\le (\sqrt{1-c_3})^{n-k},$$ 
where $\Bu=(f_1+x_1,\dots,f_n+x_n)$, $f_i$ are fixed and $x_i$ are iid copies of $\xi$.
\end{lemma}

\begin{proof}(of Lemma \ref{lemma:rankincrease}) View $M_{n+1}$ as the matrix obtained by adding the first row and first column to $M_n$. Let $H$ be the vector space of dimension $k$ spanned by the row vectors of $M_n$. Then the probability that the subvector formed by the last $n$ components of the first row of $M_{n+1}$ does not belong to $H$, by Lemma \ref{lemma:Odlyzko}, is at least $1-(\sqrt{1-c_3})^{n-k}$. 

Hence, 

$$\P(\rank(M_{n+1})= k+2 |\rank(M_n)= k) \ge 1-(\sqrt{1-c_3})^{n-k}.$$

In general, for $1\le t\le n-k$ we have

$$\P(\rank(M_{n+t}) = k+2t | \rank(M_{n+t-1})= k+2(t-1)) \ge 1-(\sqrt{1-c_3})^{n-t-k+1}.$$
 
Because the rows (and columns) added to $M_{n+t-1}$ each step (to create $M_{n+t}$) are independent, we have

\begin{align*} 
&\qquad \P(\rank(M_{2n-k-1})=2n-k-2 |\rank(M_n)=k) \ge\\ 
&\ge \prod_{t=1}^{n-k-1} \P\big(\rank(M_{n+t}) = k+2t |\rank(M_{n+t-1})= k+2(t-1)\big)  \\
&\ge (1-(\sqrt{1-c_3})^{n-k})(1-(\sqrt{1-c_3})^{n-k-1})\cdots (1-(\sqrt{1-c_3}))=\Omega_{c_3}(1).
\end{align*}

\end{proof}

Next we show that in the case of $M_n$ having rank $n-1$, it suffices to assume that $\rank(M_{n-1})\ge n-2$, thanks to the following simple observation.

\begin{lemma}\label{lemma:optional}
Assume that $M_n$ has rank $n-1$. Then there exists $1\le i \le n$ such that the removal of the $i$-th row and the $i$-column of $M_{n}$ results in a symmetric matrix $M_{n-1}$ of rank at least $n-2$.
\end{lemma}

\begin{proof} (of Lemma \ref{lemma:optional})
With out loss of generality, assume that the last $n-1$ rows of $M_n$ span a subspace of dimension $n-1$. Then the matrix obtained from $M_n$ by removing the first row and the first column has rank at least $n-2$.
\end{proof}
 
Without loss of generality, we assume that the matrix $M_{n-1}$ obtained from $M_n$ by removing its first row and first column has rank at least $n-2$. We next express $\det(M_n)$ as a quadratic function of its first row $(m_{11},\dots,m_{1n})$ as follows.

$$\det(M_n)= c_{11}(M_n)m_{11}^2+\sum_{2\le i,j\le n}c_{ij}(M_{n-1})m_{1i}m_{1j}$$

where $c_{11}(M_n)$ is the first cofactor of $M_n$, while $c_{ij}(M_{n-1})$ are the corresponding cofactors of the matrix $M_{n-1}$. 

It is crucial to note that, since $M_{n-1}$ has rank at least $n-2$, at least one of the cofactor $c_{ij}(M_{n-1})$ is nonzero. Set $c:=(\sum_{2\le i,j \le n} c_{ij}(M_{n-1})^2)^{1/2}$ and $a_{ij}:=c_{ij}(M_{n-1})/c$.

Roughly speaking, our approach consists of two main steps.

\begin{itemize}
\item {\it Step 1}. Assume that 

$$\P_{x_{11},\dots,x_{1n}}((c_{11}(M_n)/c)m_{11}^2+\sum_{2\le i,j\le n}a_{ij}m_{1i}m_{1j}=0 |M_{n-1})\ge n^{-B},$$ 

then there is a strong additive structure among the cofactors $c_{ij}(M_{n-1})$ of $M_{n-1}$.

\vskip .2in
\item {\it Step 2}. The probability, with respect to $M_{n-1}$, that there is a strong additive structure among the $c_{ij}(M_{n-1})$ is negligible.
\end{itemize}

We will execute Step 1 by proving Theorem \ref{theorem:step1} below (as a special case). Step 2 will be carried out by proving Theorem \ref{theorem:step2}.

{\bf Case 2.} $\det(M_n)\neq 0$. Let $C(M_n)=(c_{ij}(M_n))$ be the matrix of the cofactors of $M_n$, we have

$$C(M_n)\By = \det(M_n) \cdot \Bx.$$ 

Thus $\|C(M_n)\By\| = |\det(M_n)|$. 

We infer that, without loss of generality, 

$$|c_{11}(M_n)y_1+\dots c_{1n}(M_n)y_n|\ge |\det(M_n)|/n^{1/2}.$$

Note that $\|\By\|\le n^{-A}$, thus

\begin{equation}\label{eqn:intro:1}
\sum_{j=1}^n |c_{1j}(M_n)|^2 \ge n^{2A-1} \det(M_n)^2.
\end{equation}

For $j\ge 2$, we write 

$$c_{1j}(M_n)=\sum_{i=2}^n m_{i1}c_{ij}(M_{n-1}),$$

where $M_{n-1}$ is the matrix obtained from $M_n$ by removing its first row and first column, and $c_{ij}(M_{n-1})$ are the corresponding cofactors of $M_{n-1}$.

Hence, by the Cauchy-Schwarz inequality, by Condition \ref{condition:bound}, and by the bounds $f_{ij}\le n^\gamma$ for the entries of $F_n$, we have

\begin{eqnarray}\label{eqn:intro:2}
c_{1j}(M_n)^2 &\le \sum_{i=2}^n m_{i1}^2 \sum_{i=2}^n c_{ij}^2(M_{n-1}) \nonumber\\
&\le n^{2B+2\gamma + 3}  \sum_{i=2}^n c_{ij}^2(M_{n-1}).
\end{eqnarray}

Similarly, for $j=1$ we write

$$c_{11}(M_n)=\sum_{i=2}^n m_{i2}c_{i2}(M_{n-1}).$$ 

Thus,

\begin{equation}\label{eqn:intro:3}
c_{11}(M_n)^2\le n^{2B+2\gamma+3} \sum_{i=2}^n c_{i2}^2(M_{n-1}).
\end{equation}

It follows from \eqref{eqn:intro:1},\eqref{eqn:intro:2} and \eqref{eqn:intro:3} that

$$\sum_{2\le i,j \le n} c_{ij}(M_{n-1})^2 \ge n^{2A-2B-2\gamma-4}\det(M_{n})^2.$$

Hence, for proving Theorem \ref{theorem:singularvalue}, it suffices to justify the following result.

\begin{theorem}\label{theorem:singularvalue''}
For any $B>0$, there exists $A>0$ such that 

$$\P\big((\sum_{2\le i,j \le n} c_{ij}(M_{n-1})^2)^{1/2} \ge n^A|\det(M_n)| \big)\le n^{-B}.$$
\end{theorem}

To prove Theorem \ref{theorem:singularvalue''}, we again express $\det(M_n)$ as a quadratic form of its first row. 

$$\det(M_n) = c_{11}(M_n)m_{11}^2 + \sum_{2\le i,j \le n} c_{ij}(M_{n-1})m_{1i}m_{j1}.$$

In other words,

$$\det(M_n)/c =m_{11}c_{11}/c+\sum_{2\le i,j\le n}a_{ij}m_{1i}m_{1j},$$
 
where $c:=(\sum_{2\le i,j \le n} c_{ij}(M_{n-1})^2)^{1/2}$ and $a_{ij}:=c_{ij}(M_{n-1})/c.$

Roughly speaking, our approach in this case also consists of two main steps.
  
\begin{itemize}
\item {\it Step 1}. Assume that 

$$\P_{x_{11},\dots,x_{1n}}(|(c_{11}(M_n)/c)m_{11}^2+\sum_{2\le i,j\le n}a_{ij}m_{1i}m_{1j}|\le n^{-A}|M_{n-1})\ge n^{-B},$$ 

then there is a strong additive structure among the cofactors $c_{ij}$.

\vskip .2in

\item {\it Step 2}. The probability, with respect to $M_{n-1}$, that there is a strong additive structure among the $c_{ij}$ is negligible.
\end{itemize}

We now state our main supporting lemmas. 

\begin{theorem}[Step 1]\label{theorem:step1}
Let $0<\ep<1$ be given constant. Assume that $\sum_{i,j} a_{ij}^2=1$, where $a_{ij}=a_{ji}$, and 

$$\sup_a \P_{x_2,\dots,x_{n}}(|\sum_{2\le i,j\le n} a_{ij}(x_i+f_i)(x_j+f_j)-a|\le n^{-A})\ge n^{-B}$$
 
for some sufficiently large integer $A$, where $x_i$ are iid copies of $\xi$. Then there exists a vector $\Bu=(u_1,\dots,u_{n-1})$ satisfying the following properties.

\begin{itemize}
\item $\|\Bu\|\asymp 1$ and  $|\langle \Bu,\row_i(M_{n-1})\rangle| \le n^{-A/2+O_{B,\ep}(1)}$ for $n-O_{B,\ep}(1)$ rows of $M_{n-1}$.
\vskip .2in
\item There exists a generalized arithmetic progression $Q$ of rank $O_{B,\ep}(1)$ and size $n^{O_{B,\ep}(1)}$ that contains at least $n-2n^\ep$ components $u_i$.
\vskip .2in
\item All the components $u_i$, and all the generators of the generalized arithmetic progression are rational numbers of the form $p/q$, where $|p|,|q| \le n^{A/2+O_{B,\ep}(1)}$.
\end{itemize}

\end{theorem}

\vskip .2in

We refer the reader to Section \ref{section:ILOquadratic} for a definition of generalized arithmetic progression. Theorem \ref{theorem:step1}, which is at the heart of the paper, follows from our study of the inverse Littlewood-Offord problem for quadratic forms (Sections \ref{section:ILOquadratic}-\ref{section:ILOquadratic:proof}).

We next proceed to the second step of the approach showing that the probability for $M_{n-1}$ having the above properties is negligible.

\begin{theorem}[Step 2]\label{theorem:step2}
With respect to $M_{n-1}$, the probability that there exists a vector $\Bu$ as in Theorem \ref{theorem:step1} is $\exp(-\Omega(n))$.
\end{theorem}

The rest of the paper is organized as follows. We state and prove the inverse Littlewood-Offord result for quadratic forms throughout Sections \ref{section:ILOquadratic}-\ref{section:ILOquadratic:proof}. As an application, we prove Theorem \ref{theorem:step1} in Section \ref{section:step1} and conclude Theorem \ref{theorem:step2} in Section \ref{section:step2}. We will prove Corollary \ref{cor:concentration} in Section \ref{section:cor}.

\section{The inverse Littlewood-Offord result for quadratic forms}\label{section:ILOquadratic}
A classical result of Erd\H{o}s \cite{E} and Littlewood-Offord \cite{LO} asserts that if $a_i$ are real numbers of magnitude $|a_i|\ge 1$, then the probability that the random sum $\sum_{i=1}^n a_ix_i$ concentrates on an interval of length one is of order $O(n^{-1/2})$, where $x_i$ are iid copies of a Bernoulli random variable. This remarkable inequality has generated an impressive way of research, particularly from the early 1960s to the late 1980s. We refer the reader to \cite{EM,Ess,H,Kan,Kat,Kle,Kol,Rog,SSz,Stan,TVbook} for further reading regarding these developments.

Motivated by inverse theorems from additive combinatorics (see \cite[Chapter 5]{TVbook}), Tao and Vu brought a new view to the problem: find the underlying reason as to why the concentration probability of $\sum_{i=1}^n a_ix_i$ on a short interval is large. 

Typical examples of $a_i$ that have large concentration probability are \emph{generalized arithmetic progressions} (GAPs). 

A set $Q$ is a \emph{GAP of rank $r$} if it can be expressed as in the form
$$Q= \{g_0+ k_1g_1 + \dots +k_r g_r| K_i \le k_i \le K_i' \hbox{ for all } 1 \leq i \leq r\}$$ for some $\{g_0,\ldots,g_r\},\{K_1,\ldots,K_r\},\{K'_1,\ldots,K'_r\}$.

\vskip .1in

It is convenient to think of $Q$ as the image of an integer box $B:= \{(k_1, \dots, k_r) \in \Z^r| K_i \le k_i \le K_i' \} $ under the linear map
$$\Phi: (k_1,\dots, k_r) \mapsto g_0+ k_1g_1 + \dots + k_r g_r. $$
The numbers $g_i$ are the \emph{generators } of $P$, the numbers $K_i'$ and $K_i$ are the \emph{dimensions} of $P$, and $\Vol(Q) := |B|$ is the \emph{volume} of $B$. We say that $Q$ is \emph{proper} if this map is one to one, or equivalently if $|Q| = \Vol(Q)$.  For non-proper GAPs, we of course have $|Q| < \Vol(Q)$.
If $-K_i=K_i'$ for all $i\ge 1$ and $g_0=0$, we say that $Q$ is {\it symmetric}.

A closer look at the definition of GAPs reveals that if $a_i$ are very {\it near} to the elements of a $GAP$ of rank $O(1)$ and size $n^{O(1)}$, then the probability that $\sum_{i=1}^n a_ix_i$ concentrates on a short interval is of order $n^{-O(1)}$, where $x_i$ are iid copies of a Bernoulli random variable. 

It was shown by Tao and Vu \cite{TVbull,TVcir,TVcomp}, in an implicit way, that these are essentially the only examples that have high concentration probability. An explicit and optimal version has been given in a recent paper by the current author and Vu.

We say that $a$ is $\delta$-close to a set $Q$ if there exists $q\in Q$ such that $|a-q|\le \delta$.

\begin{theorem}[Inverse Littlewood-Offord theorem for linear forms, \cite{NgV}]\label{theorem:ILOlinear:continuous} Let $0 <\ep < 1$ and $B>0$. Let 
$ \beta >0$ be a parameter that may depend on $n$. Suppose that not all $a_i$ are zero, and 
 
$$\sup_a\P_{\Bx} (|\sum_{i=1}^n a_i(x_i+f_i)-a| \le \beta)=\rho \ge n^{-B},$$ 

where $\Bx=(x_1,\dots,x_n)$, and $x_i$ are iid copies of a random variable $\xi$ satisfying Condition \ref{condition:space}. Then the following holds. For any number $n'$ between $n^\ep$ and $n$, there exists a proper symmetric GAP $Q=\{\sum_{i=1}^r x_ig_i : |x_i|\le L_i \}$ such that

\begin{itemize}

\item At least $n-n'$ elements of $a_i$ are $\beta$-close to $Q$.

\vskip .1in

\item $Q$ has small rank, $r=O_{B,\ep}(1)$, and small cardinality

$$|Q| \le \max \left(O_{B,\ep}(\frac{\rho^{-1}}{\sqrt{n'}}),1\right).$$

\vskip .1in

\item There is a non-zero integer $p=O_{B,\ep}(\sqrt{n'})$ such that all
 steps $g_i$ of $Q$ have the form  $g_i=(\beta/\sqrt{\sum_{i=1}^n a_i^2} )\frac{p_{i}} {p} $, with $p_{i} \in \Z$ and $p_{i}=O_{B,\ep}(\beta^{-1} \sqrt{\sum_{i=1}^n a_i^2}\sqrt{n'}).$

\end{itemize}
\end{theorem}

In this and all subsequent theorems, the hidden constants could also depend on $c_1,c_2,c_3$ of Condition \ref{condition:space}. We could have written $O_{c_1,c_2,c_3}(.)$ everywhere, but these notations are somewhat cumbersome, and this dependence is not our focus, so we omit them.

Theorem \ref{theorem:ILOlinear:continuous} was proven in \cite{NgV} with $c_1=1,c_2=2$ and $c_3=1/2$, but the proof there automatically extends to any constants $0<c_1<c_2$ and $0<c_3$.

An immediate corollary of Theorem \ref{theorem:ILOlinear:continuous} is the $\beta-$net Theorem by Tao and Vu \cite{TVcir}, which plays a crucial rule in their resolution of the circular law conjecture in random matrix theory.

To prove Theorem \ref{theorem:step1}, we need an inverse-type result for the high concentration of the quadratic form $\sum_{i}a_{ij}(x_i+f_i)(x_j+f_j)$. A classical approach is to pass the problem to the corresponding bilinear form.

\begin{theorem}[Inverse Littlewood-Offord theorem for bilinear forms]\label{theorem:ILObilinear:continuous}
Let $0 <\ep < 1$ and $B>0$. Let $ \beta >0$ be a parameter that may depend on $n$. Suppose that $\sum_{i,j}^n a_{ij}^2 =1$ and 

$$\sup_a \P_{\Bx,\By}(|\sum_{i,j\le n}a_{ij}(x_i+f_i)(y_j+f_j)-a|\le \beta)=\rho \ge n^{-B},$$
 
where $\Bx=(x_1,\dots,x_n), \By=(y_1,\dots,y_n)$, and $x_i$ and $y_i$ are iid copies of a random variable $\xi$ satisfying Condition \ref{condition:space}, and $f_i$ are fixed. Then there exist a set $I_0$ of size $O_{B,\ep}(1)$, a set $I$ of size at least $n-2n^\ep$, and integers $0\neq k,k_{ii_0},i_0\in I_0,i\in I$, all bounded by $n^{O_{B,\ep}(1)}$ such that the following holds. Let $R$ be the matrix defined as

\begin{equation}
(R)_{ij}:= 
\begin{cases}
k_{ij}, & \text{ if $j\in I_0$;}\\ 
k, & \text{ if $j=i$;}\\
0, & \text{ otherwise}
\end{cases}
\end{equation}

for each $i\in I$; and the other entries $k_{ij}$ of $R$ are zero, except the diagonal terms $k_{ii}, i\notin I$ which are defined to be one. Then the following holds for any $i$-th row, $i\in I$, of the matrix $A'=RA$,

\begin{equation}\label{eqn:special}
\P_\By(|\langle \By+\Bf, \row_i(A')\rangle|\le \beta n^{O_{B,\ep}(1)})\ge n^{-O_{B,\ep}(1)},
\end{equation}

where $\Bf=(f_1,\dots,f_n)$.
\end{theorem}

One may apply Theorem \ref{theorem:ILOlinear:continuous} to \eqref{eqn:special} to conclude that most of the components of $\row_i(A')$ are very close to a GAP of rank $O(1)$ and $n^{O(1)}$. However, we prefer to keep \eqref{eqn:special} as it is more convenient to use. 

Roughly speaking, Theorem \ref{theorem:ILObilinear:continuous} states that if $\sup_a \P_{\Bx,\By}(|\sum_{i,j\le n}a_{ij}(x_i+f_i)(y_j+f_j)-a|\le \beta) \ge n^{-B}$, then there exist $r=|I_0|=O(1)$ rows $\row_{i_1}(A),\dots,\row_{i_r}(A)$ of $A$ such that most of the remaining rows of $A$ can be written as in the form $k_1\row_{i_1}(A)+\dots+k_{r}\row_{i_r}(A)+\Ba$, where $|k_i|=n^{O(1)}$, and most of the components of $\Ba$ are very close to a GAP of rank $O(1)$ and size $n^{O(1)}$.

\vskip .2in

We now state our main inverse-type result for high concentration probability of quadratic forms.

\begin{theorem}[Inverse Littlewood-Offord theorem for quadratic forms]\label{theorem:ILOquadratic:continuous}
Let $0 <\ep < 1$ and $B>0$. Let $ \beta >0$ be a parameter that may depend on $n$. Suppose that $\sum_{i=1}^n a_{ij}^2 =1$, where $a_{ij}=a_{ji}$, and 

$$\sup_a \P_{\Bx}(|\sum_{i,j\le n}a_{ij}(x_i+f_i)(x_j+f_j)-a|\le \beta)=\rho \ge n^{-B}.$$ 

Then there exists a matrix $R$ as in Theorem \ref{theorem:ILObilinear:continuous} such that for any $i\in I$,

$$\P_\Bz(|\langle \Bz, \row_i(A')\rangle|\le \beta n^{O_{B,\ep}(1)})\ge n^{-O_{B,\ep}(1)},$$ 

where $\Bz=(z_1,\dots, z_n)$ and $z_i$ are iid copies of $\eta^{(1/2)} (\xi-\xi')$, where $\eta^{(1/2)}$ is a Bernoulli random variable of parameter $1/2$ independent of $\xi$ and $\xi'$.

\end{theorem}

\begin{remark}
All the results presented in this section are in continuous setting. Similar results in discrete setting were obtained in \cite{Ng}. The reader may also consult that paper for more examples and motivations.
\end{remark}

\section{A rank reduction argument and the full rank assumption}\label{appendix:fullrank}
This section, which is independent of its own, provides a technical lemma we will need for later sections. Informally, it says that if we can find a proper symmetric GAP that contains a given set, then we can assume this containment is non-degenerate.   

Assume that $P=\{k_1g_1+\dots+k_rg_r | -K_i\le k_i \le K_i\}$ is a proper symmetric GAP, which contains a set $U=\{u_1,\dots. u_n\}$. 

We consider $P$ together with the map $\Phi: P \rightarrow \R^r$ which maps $k_1g_1+\dots+k_rg_r$ to $(k_1,\dots,k_r)$. Because $P$ is proper, this map is bijective. 

We know that $P$ contains $U$, but we do not know yet that $U$ is non-degenerate in $P$ in the sense that the set $\Phi(U)$ has full rank in $\R^{r}$. In the later case, we say $U$ {\it spans} P.

\begin{theorem}\label{theorem:fullrank}
Assume that $U$ is a subset of a proper symmetric GAP $P$ of size $r$, then there exists a proper symmetric GAP $Q$ that contains $U$ such that the followings hold.

\begin{itemize}
\item $\rank(Q)\le r$ and $|Q|\le O_r(1)|P|$.

\vskip .1in

\item $U$ spans $Q$, that is, $\phi(U)$ has full rank in $\R^{\rank(Q)}$.
\end{itemize}

\end{theorem}

To prove Theorem \ref{theorem:fullrank}, we will rely on the following lemma.

\begin{lemma}[Progressions lie inside proper progressions]\label{lemma:embeding}
There is an absolute constant $C$ such that the following holds. Let $P$ be a GAP of rank $r$ in $\R$. Then there is a symmetric proper GAP $Q$ of rank at most $r$ containng $P$ and 

$$|Q|\le r^{Cr^3}|P|.$$ 

\end{lemma}

\begin{proof} (of Theorem \ref{theorem:fullrank}) We shall mainly follow \cite[Section 8]{TVsing}.

Suppose that $\Phi(U)$ does not have full rank, then it is contained in a hyperplane of $\R^r$. In other words, there exist integers $\alpha_1,\dots,\alpha_r$ whose common divisor is one and $\alpha_1k_1+\dots + \alpha_r k_r=0$ for all $(k_1,\dots,k_r)\in \Phi(U)$.

Without loss of generality, we assume that $\alpha_r \neq 0$. We select $w$ so that $g_r=\alpha_r w$, and consider $P'$ be the GAP generated by $g_i':=g_i-\alpha_iw$ for $1\le i \le r-1$. The new symmetric GAP $P'$ will continue to contain $U$, because we have 

\begin{align*}
k_1g_1'+\dots +k_{r-1}g_{r-1}' &= k_1g_1+\dots+k_rg_r - w(\alpha_1k_1+\dots+\alpha_rg_r)\\
&= k_1g_1+\dots+k_rg_r
\end{align*}

for all $(k_1,\dots,k_r)\in \Phi(U)$. 

Also, note that the volume of $P'$ is $2^{r-1}K_1\dots K_{r-1}$, which is less than the volume of $P$.

We next use Lemma \ref{lemma:embeding} to guarantee that $P'$ is symmetric and proper without increasing the rank. 
 
Iterate the process if needed. Because the rank of the newly obtained proper symmetric GAP decreases strictly after each step, the process must terminate after at most $r$ steps.
 
\end{proof}

\section{proof of Theorem \ref{theorem:ILObilinear:continuous}}\label{section:ILObilinear:proof}

For minor technical reasons, it is convenient to assume $\xi$ to be a discrete random variable. The continuous case of the theorem can be recovered from the discrete one by a standard argument (approximating the continuous distribution by a discrete one while holding $n$ fixed).  

For short, we denote the vector $(a_{i1},\dots,a_{in})$ by $\Ba_i$. We begin by applying Theorem \ref{theorem:ILOlinear:continuous}.

\begin{lemma}\label{lemma:roworthogonal} 
Let $\ep<1$, and $B$ be positive constants. Assume that 

$$\rho=\sup_a\P_{\Bx,\By} (|\sum_{i,j}a_{ij}(x_i+f_i)(y_j+f_j) -a|\le \beta)\ge n^{-B}.$$ 

Then the following holds with probability at least $3\rho/4$ with respect to $\By=(y_1,\dots,y_n)$. There exist a proper symmetric GAP $Q_\By$ of rank $O_{B,\ep}(1)$ and size $\max(O_{B,\ep}(\rho^{-1}/n^{\ep/2}),1)$ and a set $I_\By$ of $n-n^\ep$ indices such that $\langle \Ba_i, \By \rangle$ is $\beta$-close to $Q_\By$ for each $i\in I_{\By}$.
\end{lemma}

\begin{proof}(of Lemma \ref{lemma:roworthogonal}) 
For short we write

$$\sum_{i,j} a_{ij}(x_i+f_i)(y_j+f_j) = \sum_{i=1}^n (x_i+f_i) \langle \Ba_i,(\By+\Bf) \rangle .$$

We say that a vector $\By=(y_1,\dots,y_n)$ is {\it good} if

$$\P_\Bx(|\sum_{i=1}^n (x_i+f_i)\langle \Ba_i,(\By+\Bf) \rangle -a|\le \beta)\ge \rho/4.$$
 
We call $\By$ {\it bad} otherwise. 

Let $G$ denote the collection of good vectors. We are going to estimate the probability $p$ of a randomly chosen vector $\By=(y_1,\dots,y_n)$ being bad by an averaging method.

\begin{align*}
\P_{\By} \P_{\Bx} (|\sum_{i=1}^n (x_i+f_i)\langle \Ba_i,(\By+\Bf) \rangle-a|\le \beta) &=\rho\\
p \rho/4 + 1-p &\ge \rho\\
(1-\rho)/(1-\rho/4) &\ge p.
\end{align*}

Thus, the probability of a randomly chosen $\By$ belonging to $G$ is at least 

$$1-p \ge (3\rho/4)/(1-\rho/4) \ge 3\rho/4.$$

Consider a good vector $\By\in G$. By definition, we have

$$\P_{\Bx} (|\sum_{i=1}^n (x_i+f_i) \langle \Ba_i,(\By+\Bf) \rangle -a|\le \beta) \ge \rho/4.$$

Thus 

$$\sup_b \P_\Bx(|\sum_{i=1}^n x_i\langle \Ba_i, (\By+\Bf) \rangle -b|\le \beta) \ge \rho/4.$$

Next, if $\langle \Ba_i,(\By+\Bf) \rangle =0$ for all $i$, then the conclusion of the theorem holds trivially. Otherwise, we apply Theorem \ref{theorem:ILOlinear:continuous} to the sequence $\langle \Ba_i,(\By+\Bf) \rangle$, $i=1,\dots,n$, obtaining a proper symmetric GAP $Q_{\By}$ of rank $O_{B,\ep}(1)$ and size $\max(O_{B,\ep}(\rho^{-1}/n^{\ep/2}),1)$, together with its elements $q_i(\By)$ such that $|\langle \Ba_i,\By\rangle -q_i(\By)|\le \beta$ for all $i\in I_\By$, a set of at least $n-n^\ep$.
\end{proof}

We now work with the $q_i(\By)$. By Theorem \ref{theorem:fullrank}, we may assume that the $q_i(\By)$ span $Q_{\By}$.

Recall that

\begin{equation}
\P_\By(\By\in G)\ge 3\rho/4.
\end{equation}

Next, for each $\By\in G$, we choose from $I_\By$ $s$ indices $i_{y_1},\dots,i_{y_s}$ such that $q_{i_{y_j}}(\By)$ span $Q_\By$, where $s$ is the rank of $Q_\By$. We note that $s=O_{B,\ep}(1)$ for all $s$. 

Consider the tuples $(i_{y_1},\dots,i_{y_s})$ for all $\By\in G$. Because there are $\sum_{s} O_{B,\ep}(n^s) = n^{O_{B,\ep}(1)}$ possibilities these tuples can take, there exists a tuple, say $(1,\dots,r)$ (by rearranging the rows of $A$ if needed) such that $(i_{y_1},\dots,i_{y_s})=(1,\dots,r)$ for all $\By\in G'$, and a subset $G'$ of $G$ satisfying 

\begin{equation}
\P_\By(\By\in G')\ge \P_\By(\By\in G)/n^{O_{C,\ep}(1)} =\rho/n^{O_{B,\ep}(1)}.
\end{equation}

For each $1\le i\le r$, we express $q_i(\By)$ in terms of the generators of $Q_\By$ for each $\By\in G'$, 

$$q_i(\By) = c_{i1}(\By)g_{1}(\By)+\dots + c_{ir}(\By)g_{r}(\By),$$ 

where $c_{i1}(\By),\dots c_{ir}(\By)$ are integers bounded by $n^{O_{B,\ep}(1)}$, and $g_{i}(\By)$ are the generators of $Q_\By$.

We will show that there are many $\By$ that correspond to the same coefficients $c_{ij}$. 

Consider the collection of the coefficient-tuples $\Big(\big(c_{11}(\By),\dots,c_{1r}(\By)\big);\dots; \big(c_{r1}(\By),\dots c_{rr}(\By)\big)\Big)$ for all $\By\in G'$. Because the number of possibilities these tuples can take is at most

$$(n^{O_{B,\ep}(1)})^{r^2} =n^{O_{B,\ep}(1)}.$$

There exists a coefficient-tuple, say  $\Big((c_{11},\dots,c_{1r}),\dots, (c_{r1},\dots c_{rr})\Big)$, such that  

$$\Big(\big(c_{11}(\By),\dots,c_{1r}(\By)\big);\dots; \big(c_{r1}(\By),\dots c_{rr}(\By)\big)\Big) =\Big((c_{11},\dots,c_{1r}),\dots, (c_{r1},\dots c_{rr})\Big)$$  

for all $\By\in G''$, a subset of $G'$ satisfying 

\begin{equation}
\P_\By(\By\in G'')\ge \P_\By(\By\in G')/n^{O_{B,\ep}(1)} \ge \rho/n^{O_{B,\ep}(1)}.
\end{equation}

In summary, there exist $r$ tuples  $(c_{11},\dots,c_{1r}),\dots, (c_{r1},\dots c_{rr})$, whose components are integers bounded by $n^{O_{B,\ep}(1)}$, such that the followings hold for all $\By\in G''$.

\begin{itemize}

\item $q_i(\By) = c_{i1}g_{1}(\By)+\dots + c_{jr}g_{r}(\By)$, for $i=1,\dots,r$.

\vskip .1in

\item The vectors  $(c_{11},\dots,c_{1r}),\dots, (c_{r1},\dots c_{rr})$ span $\Z^{\rank(Q_\By)}$.

\end{itemize}

Next, because $|I_\By|\ge n-n^\ep$ for each $\By\in G''$, by an averaging argument, there is a set $I$ of size at least $n-2n^\ep$ such that for each $i\in I$ we have 

\begin{equation}
\P_\By(i\in I_\By, \By\in G'') \ge \P_\By(\By\in G'')/2.
\end{equation}  

Fix an arbitrary row $\Ba$ of index from $I$. We concentrate on those $\By\in G''$ where the index of $\Ba$ belongs to $I_\By$. 

Because $q(\By) \in Q_{\By}$ ($q(\By)$ is the element from $Q_\By$ that is $\beta$-close to $\langle \Ba, \By \rangle$), we can write 

$$q(\By) = c_{1}(\By)g_{1}(\By)+\dots c_{r}(\By)g_{r}(\By)$$ 

where $c_{i}(\By)$ are integers bounded by $n^{O_{B,\ep}(1)}$.

For short, for each $i$ we denote by $\Bv_i$ the vector $(c_{i1},\dots,c_{ir})$, we will also denote by $\Bv_{\Ba,\By}$ the vector $(c_{1}(\By),\dots c_{r}(\By))$. 

Because $Q_{\By}$ is spanned by $q_1(\By),\dots, q_{r}(\By)$, we have $k=\det(\mathbf{v}_1,\dots \mathbf{v}_r)\neq 0$, and that 

$$k q(\By) + \det(\mathbf{v}_{\Ba,\By},\mathbf{v}_2,\dots,\mathbf{v}_r)q_{1}(\By) +\dots + \det(\mathbf{v}_{\Ba,\By},\mathbf{v}_1,\dots,\mathbf{v}_{r-1})q_{r}(\By) =0.$$ 

Next, because each coefficient of the identity above is bounded by $n^{O_{B,\ep}(1)}$, there exists a subset $G_{\Ba}''$ of $G''$ such that all $\By\in G_{\Ba}''$ correspond to the same identity, and

\begin{equation}
\P_\By(\By\in G_{\Ba}'') \ge (\P_\By(\By\in G'')/2)/(n^{O_{B,\ep}(1)})^r = \rho/n^{O_{B,\ep}(1)}.
\end{equation}

In other words, there exist integers $k_1,\dots,k_r$ depending on $\Ba$, all bounded by $n^{O_{B,\ep}(1)}$, such that 

\begin{equation}\label{eqn:ILObilinear:q}
k q(\By) + k_1 q_{1}(\By) + \dots + k_r q_{r}(\By)=0
\end{equation}

for all $\By\in G_{\Ba}''$. 

Note that $k$ is independent of $\Ba$ and $\By$. It is crucial to note that, since $q_i(\By)$ is $\beta$-close to $\langle \Ba_i,(\By+\Bf) \rangle$. It follows from \eqref{eqn:ILObilinear:q} that for all $\By\in G_{\Ba}''$

\begin{equation}\label{eqn:ILObilinear:a}
|k \langle \Ba, (\By+\Bf)\rangle + k_1  \langle \Ba_1, (\By+\Bf) \rangle + \dots + k_r \langle \Ba_{r},(\By+\Bf) \rangle|\le n^{O_{B,\ep}(1)}\beta.
\end{equation}

Next we introduce the following matrix.

\begin{definition}[Row matrix $R$] For each $i\in I$, consider $\Ba=\row_i(A)$ and those $k_1,\dots,k_r$ from \eqref{eqn:ILObilinear:a}. We define the $i$-th row of $R$ as

\begin{equation}\label{eqn:ILObilinear:R}
\row_i(R):= (k_{1},\dots,k_{r},0,\dots,0,k,0,\dots,0).
\end{equation}

The other entries of $R$ are zero, except the diagonal terms $(R)_{ii}$, where $i\notin I$, which are set to be one.

\end{definition}

It is clear from the definition of $R$ that 

\begin{equation}\label{eqn:ILObilinear:R:bound}
n^{-O_{B,\ep}(1)} \le \sigma_n(R) \le \sigma_1(R)\le n^{O_{B,\ep}(1)}.
\end{equation}

Let $A':=RA$. It follows from  \eqref{eqn:ILObilinear:a} that 

\begin{equation}\label{eqn:ILObilinear:final}
\P_\By(|\langle \By+\Bf, \row_i(A')\rangle|\le \beta n^{O_{B,\ep}(1)})\ge \rho/n^{O_{B,\ep}(1)}=n^{-O_{B,\ep}(1)}
\end{equation}

for any $i\in I$, hence obtaining the conclusion of the theorem for the case $\xi$ having discrete distribution.

To recover the continuous case, we approximate $\xi$ by a discrete distribution while holding $\beta,n,f_i'$ and $a_{ij}$ fixed. By taking the limit, we again obtain \eqref{eqn:ILObilinear:final} for some $R$.

\section{proof of Theorem \ref{theorem:ILOquadratic:continuous}}\label{section:ILOquadratic:proof}

In this section we will use the results from Section \ref{section:ILObilinear:proof} to prove Theorem \ref{theorem:ILOquadratic:continuous}.

Let $U$ be a random subset of $\{1,\dots,n\}$, where $\P(i\in U)=1/2$ for each $i$. Let $A_U$ be a matrix of size $n$ by $n$ defined by 

\[
(A_U)_{ij}= 
\begin{cases}
a_{ij} & \text{ if either $i\in U$ and $j\notin  U$, or $i\notin U$ and $j\in U$},\\
0 & \text{otherwise.}
\end{cases}
\]

We first apply the following lemma.

\begin{lemma}[Concentration for bilinear forms controls concentration for quadratic forms]\label{lemma:decoupling}
Let $U$ be given. Assume that 

$$\sup_a \P_{\Bx} (|\sum_{i,j}a_{ij}(x_i+f_i)(x_j+f_j)-a|\le \beta)=\rho \ge n^{-B},$$

where $\Bx=(x_1,\dots,x_n)$ and $x_i$ are iid copies of $\xi$. 

Then

\begin{equation}\label{eqn:ILObilinear:0}
\P_{\Bv,\Bw}(|\sum_{i,j} A_U(ij)v_iw_j|=O_B(\beta \sqrt{\log n})) \ge \rho^8/((2\pi)^{7/2}\exp(4\pi)),
\end{equation}

where $\Bv=(v_1,\dots,v_n)$, $\Bw=(w_1,\dots,w_n)$, and $v_i,w_j$ are iid copies of  $\xi-\xi'$.
\end{lemma}

We refer the reader to Appendix \ref{appendix:decoupling} for a proof of this lemma.

By the definition of $\xi$, it is clear that the random variable $\xi-\xi'$ also satisfies Condition \ref{condition:space} (with different positive parameters). We next apply Theorem \ref{theorem:ILObilinear:continuous} to \eqref{eqn:ILObilinear:0} to obtain the following lemma. 

\begin{lemma}\label{lemma:quadratic:row} There exist a set $I_0(U)$ of size $O_{B,\ep}(1)$ and a set $I(U)$ of size at least $n-n^\ep$, and a nonzero integer $k(U)$ bounded by $n^{O_{B,\ep}(1)}$ such that for any $i\in I$, there are integers $k_{ii_0}(U), i_0\in I_0(U)$, all bounded by $n^{O_{B,\ep}(1)}$, such that 

$$\P_\By\big(|\langle k(U)\Ba_i(A_U)+ \sum_{i_0\in I_0} k_{ii_0}(U) \Ba_{i_0}(A_U) ,\By \rangle| \le \beta n^{O_{B,\ep}(1)}) = n^{-O_{B,\ep}(1)},$$

where $\By=(y_1,\dots,y_n)$ and $y_i$ are iid copies of $\xi-\xi'$.
\end{lemma}

Note that Lemma \ref{lemma:quadratic:row} holds for all $U\subset [n]$. We will establish a similar result for the original matrix $A$.

As $I_0(U)\subset [n]^{O_{B,\ep}(1)}$ and $k(U)\le n$, there are only $n^{O_{B,\ep}(1)}$ possibilities that $(I_0(U),k(U))$ can take. Thus there exists a tuple $(I_0,k)$ such that 
$I_0(U)=I_0$ and $k(U)=k$ for $2^n/n^{O_{B,\ep}(1)}$ different sets $U$. Let us denote this set of $U$ by $\mathcal{U}$. Thus 

$$|\mathcal{U}|\ge 2^n/n^{O_{B,\ep}(1)}.$$

Next, let $I$ be the collection of $i$ which belong to at least $|\mathcal{U}|/2$ index sets $I_U$. Then we have
 
\begin{align*}                         
|I||\mathcal{U}| + (n-|I|)|\mathcal{U}|/2 & \ge (n-n^\ep )|\mathcal{U}|\\
|I| &\ge  n-2n^\ep.
\end{align*}

Fix an $i\in I$. Consider the tuples $(k_{ii_0}(U), i_0\in I_0)$ where $i\in I_U$. Because there are only $n^{O_{B,\ep}(1)}$ possibilities such tuples can take, there must be a tuple, say $(k_{ii_0}, i_0\in I_0)$, such that $(k_{ii_0}(U), i_0\in I_0)=(k_{ii_0}, i_0\in I_0)$ for at least $|\mathcal{U}|/2n^{O_{B,\ep}(1)}=2^n/n^{O_{B,\ep}(1)}$ sets $U$. 

Because $|I_0|=O_{B,\ep}(1)$, it is easy to see that there is a way to partition $I_0$ into $I_0' \cup I_0''$ such that there are $2^n/n^{O_{B,\ep}(1)}$ sets $U$ above satisfying that $I_0''\subset U$ and  $U\cap I_0'=\emptyset$. Let $\mathcal{U}_{I_0',I_0''}$ denote the collection of these $U$.

By passing to consider a subset of  $\mathcal{U}_{I_0',I_0''}$ if needed, we may assume that either $i\notin U$ or $i\in U$ for all $U\in  \mathcal{U}_{I_0',I_0''}$. Without loss of generality, we assume the first case that $i\in U$. (The other case can be treated similarly).

Let $U\in \mathcal{U}_{I_0',I_0''}$ and $\Bu=(u_1,\dots,u_n)$ be its characteristic vector ($u_j=1$ if $j\in U$, and $u_j=0$ otherwise). 

By the definition of $A_U$, and because $I_0'\cap U=\emptyset$ and $I_0''\subset U$, for $i_0'\in I_0'$ and $i_0''\in I_0''$ we can write 

$$\langle \Ba_{i_0'}(A_U),\By \rangle = \sum_{j=1}^n a_{i_0'j}u_jy_j, \mbox{ and } \langle \Ba_{i_0''}(A_U),\By \rangle = \sum_{j=1}^n a_{i_0''j}(1-u_j)y_j.$$

Also, because $i\notin U$, we have

$$\langle \Ba_{i}(A_U),\By \rangle = \sum_{j=1}^n a_{ij}u_jy_j.$$

Thus, 

\begin{align*}
&\quad \langle k\Ba_i(A_U),\By \rangle + \sum_{i_0\in I_0} \langle k_{ii_0} \Ba_{i_0}(A_U),\By \rangle\\
& = \langle k\Ba_i(A_U),\By \rangle +  \langle \sum_{i_0'\in I_0'} k_{ii_0'} \Ba_{i_0'}(A_U),\By \rangle + \langle \sum_{i_0''\in I_0''} k_{ii_0''} \Ba_{i_0''}(A_U),\By \rangle\\ 
&= \sum_{j=1}^n ka_{ij} u_jy_j + \sum_{j=1}^n \sum_{i_0'\in I_0'} k_{ii_0'} a_{i_0'j} u_jy_j +  \sum_{j=1}^n \sum_{i_0''\in I_0''} k_{ii_0''} a_{i_0''j} (1-u_j)y_j\\
&= \sum_{j=1}^n (ka_{ij} + \sum_{i_0'\in I_0'} k_{ii_0'} a_{i_0'j}- \sum_{i_0''\in I_0''} k_{ii_0''} a_{i_0''j} ) u_jy_j +  \sum_{j=1}^n \sum_{i_0''\in I_0''} k_{ii_0''} a_{i_0''j} y_j.
\end{align*}

Next, by Lemma \ref{lemma:quadratic:row}, for each $U\in \mathcal{U}_{I_0',I_0''}$ we have  

$$\P_\By\big (|\langle k \Ba_i(A_U),\By \rangle + \sum_{i_0\in I_0} \langle k_{ii_0} \Ba_{i_0}(A_U),\By \rangle| = O(\beta n^{O_{B,\ep}(1)})\big)=n^{-O_{B,\ep}(1)}.$$ 

Also, note that 

$$|\mathcal{U}_{I_0',I_0''}|= 2^n/n^{O_{B,\ep}(1)}.$$ 

Hence, 

$$\E_\By\E_U \big(k\langle \Ba_i(A_U),\By \rangle + \sum_{i_0\in I_0} \langle k_{ii_0} \Ba_{i_0}(A_U),\By \rangle =O(\beta n^{O_{B,\ep}(1)})\big) \ge n^{-O_{B,\ep}(1)}.$$

By applying the Cauchy-Schwarz inequality, we obtain 

\begin{align*}
n^{-O_{B,\ep}(1)}&\le \Big[\E_\By \E_U(k\langle \Ba_i(A_U),\By \rangle + \sum_{i_0\in I_0} \langle k_{ii_0} \Ba_{i_0}(A_U),\By \rangle =O(\beta n^{O_{B,\ep}(1)}))\Big]^2 \\
&\le  \E_\By \Big[\E_U(k\langle \Ba_i(A_U),\By \rangle + \sum_{i_0\in I_0} \langle k_{ii_0} \Ba_{i_0}(A_U),\By \rangle =O(\beta n^{O_{B,\ep}(1)}))\Big]^2 \\
&= \E_\By \Big[\E_{\Bu}(\sum_{j=1}^n (ka_{ij}+\sum_{i_0'\in I_0'} k_{ii_0'} a_{i_0'j}-\sum_{i_0''\in I_0''} k_{ii_0''} a_{i_0''j}) u_jy_j+  \sum_{j=1}^n \sum_{i_0''\in I_0''} k_{ii_0''} a_{i_0''j} y_j= O(\beta n^{O_{B,\ep}(1)}))\Big]^2\\
&\le \E_\By \E_{\Bu,\Bu'}\big(\sum_{j=1}^n (k_{ij}+\sum_{i_0'\in I_0'} k_{ii_0'} a_{i_0'j}-\sum_{i_0''\in I_0''} k_{ii_0''} a_{i_0''j}) (u_j-u_j')y_j= O(\beta n^{O_{B,\ep}(1)})\big)\\
&= \E_\Bz\big(\sum_{j=1}^n (ka_{ij}+\sum_{i_0'\in I_0'} k_{ii_0'} a_{i_0'j}-\sum_{i_0''\in I_0''} k_{ii_0''} a_{i_0''j})z_j =O(\beta n^{O_{B,\ep}(1)})\big), 
\end{align*}

where  $z_j:=(u_j-u_j')y_j$, and in the last inequality we used the simple observation that 

$$\E_{\Bu,\Bu'}\big(f(\Bu)=O(\beta n^{O_{B,\ep}(1)}),f(\Bu')=O(\beta n^{O_{B,\ep}(1)})\big) \le \E_{\Bu,\Bu'}(f(\Bu)-f(\Bu')=O(\beta n^{O_{B,\ep}(1)})).$$

Note that $u_j-u_j'$ are iid copies of $\eta^{(1/2)}$. Hence $z_j$ are iid copies of $\eta^{(1/2)}(\xi-\xi')$, where $\eta^{(1/2)}$ is independent of $\xi$ and $\xi'$.

Clearly, we may assume that $I\cap I_0=\emptyset$ by throwing away those $i\in I$ that belong to $I_0$. We introduce the following matrix $R$.

\begin{definition}

\begin{equation}\label{eqn:quadratic:R}
(R)_{ij}:= 
\begin{cases}
k_{ij}, & \text{ if $j\in I_0'$;}\\ 
-k_{ij} & \text{ if $j \in I_0''$;}\\
k, & \text{ if $j=i$;}\\
0, & \text{ otherwise}
\end{cases}
\end{equation}

for each $i\in I$; the other entries of $R$ are zero, except the diagonal terms $(R)_{ii}$, where $i\notin I$, are defined to be one.

It is clear that the matrix $A'=RA$ satisfies the conclusion of Theorem \ref{theorem:ILOquadratic:continuous}, completing the proof. 
\end{definition}

\section{proof of Theorem \ref{theorem:step1}}\label{section:step1}

We first apply Theorem \ref{theorem:ILOquadratic:continuous} to $a_{ij}$ to obtain a matrix $A'$ and a set $|I|\ge n-2n^\ep$ such that for any $i\in I$,

\begin{equation}\label{eqn:step1:0}
\P_{\Bz} (|\langle \Bz, \row_i(A') \rangle|\le n^{-A+O_{B,\ep}(1)} ) \ge n^{-O_{B,\ep}(1)}.
\end{equation}

Ideally, our next step is to apply Theorem \ref{theorem:ILOlinear:continuous} to the $\row_i(A')$. However, the application is meaningful only when $\|\row_i(A')\|$ is relatively large. Investigating the degenerate case is our next goal.

Set 

$$K=n^{-A/2}.$$

We consider two cases.

{\bf Case 1.}({\it{degenerate case}})  $\|\row_i(A')\| \le K$ for all $i\in I$. Hence,  

\begin{equation}\label{eqn:step1:1}
\|k\row_i(A)+ \sum_{j\in I_0} k_{ij} \row_j(A)\|=\|\row_i(A')\|\le K.
\end{equation} 

Next, because $\sum_j\|\col_j(A)\|^2=1$, there exists an index $j_0$ such that $\|\col_{j_0}(A)\|\ge n^{-1/2}$. Consider this column vector.

It follows from \eqref{eqn:step1:1} that for any $i\in I$,

$$|k\col_{j_0}(i)+\sum_{j\in I_0} k_{ij}\col_{j_0}(j)|\le K.$$

The above inequality means that the components $\col_{i_0}(i)$ of $\col_{j_0}(A)$ belong to a GAP generated by $\col_{j_0}(j)/k, j\in I_0$, up to an error $K$. This suggests us the following approximation.  

For each $j\notin I$, we approximate $\col_{j_0}(j)$ by a number $v_j$ of the form $(1/\lfloor 2K^{-1} \rfloor) \cdot \Z$ such that $|v_j-\col_{j_0}(j)|\le K$. We next set 

$$v_i:=\sum_{j\in I_0}k_{ij}v_j/k$$ 

for any $i\in I$. 

Thus, $v_i$ belongs to a GAP of rank $O_{B,\ep}(1)$ and size $n^{O_{B,\ep}(1)}$ for all $i\in I$.

With $\Bv=(v_1,\dots, v_{n-1})$, we have

$$\|\Bv -\col_{j_0}(A)\|\le Kn^{O_{B,\ep}(1)}.$$ 

Furthermore, by Condition \ref{condition:bound}, and because $\langle \col_{i_0}(A),\row_i(M_{n-1}) \rangle =0$ for $i\neq j_0$, we infer that

$$|\langle \Bv,\row_i(M_{n-1}) \rangle| \le Kn^{O_{B,\ep}(1)}.$$ 

Note that $\|\Bv\|\gg n^{-1/2}$. Set $\Bu:=\lfloor 1/\|\Bv\|\rfloor \cdot \Bv$, we then obtain

\begin{itemize}
\item $|\langle \Bu,\row_i(M_{n-1})\rangle| \le n^{-A/2+O_{B,\ep}(1)}$ for $n-2$ rows of $M_{n-1}$.
\vskip .2in
\item There exists a GAP of rank $O_{B,\ep}(1)$ and size $n^{O_{B,\ep}(1)}$ that contains at least $n-2n^\ep$ components $u_i$.
\vskip .2in
\item All the components $u_i$, and all the generators of the GAP are rational numbers of the form $p/q$, where $|p|,|q| \le n^{A/2+O_{B,\ep}(1)}$.
\end{itemize}

{\bf Case 2.}({\it{non-degenerate case}}). There exists $i_0\in I$ such that $\|\row_{i_0}(A')\| \ge K$. Because $\row_{i_0}(A')=k\row_{i_0}(A)+ \sum_{j\in I_0} k_{i_0j} \row_j(A)$, $\row_{i_0}(A')$ is orthogonal to $n-|I_0|-1=n-O_{B,\ep}(1)$ column vectors of $M_{n-1}$. Consequently, because $M_{n-1}$ is symmetric, $\row_{i_0}(A')$ is orthogonal to $n-O_{B,\ep}(1)$ row vectors of $M_{n-1}$.

Set 

$$\Bv:=\row_{i_0}(A')/\|\row_{i_0}(A')\|.$$

Hence, $\langle \Bv ,\row_i(M_{n-1}) \rangle =0$ for at least $n-O_{B,\ep}(1)$ row vectors of $M_{n-1}$.

Also, it follows from \eqref{eqn:step1:0} that

\begin{equation}\label{eqn:step1:2}
\P_\Bz(|\langle \Bz, \Bv \rangle|\le n^{-A/2+O_{B,\ep}(1)})\ge n^{-O_{B,\ep}(1)}.
\end{equation}

Next, because the $z_i$ satisfy Condition \ref{condition:space}, Theorem \ref{theorem:ILOlinear:continuous} applying to \eqref{eqn:step1:2} implies that $\Bv$ can be approximated by a vector $\Bu$ as follows.

\begin{itemize}
\item $|u_i-v_i|\le n^{-A/2+O_{B,\ep}(1)}$ for all $i$.
\vskip .2in
\item There exists a GAP of rank $O_{B,\ep}(1)$ and size $n^{O_{B,\ep}(1)}$ that contains at least $n-n^\ep$ components $u_i$.
\vskip .2in
\item All the components $u_i$, and all the generators of the GAP are rational numbers of the form $p/q$, where $|p|,|q| \le n^{A/2+O_{B,\ep}(1)}$.
\end{itemize}

Note that, by the approximation above, we have $\|\Bu\|\asymp 1$ and $|\langle \Bu,\row_i(M_{n-1}) \rangle| \le n^{-A/2+O_{B,\ep}(1)}$ for at least $n-O_{B,\ep}(1)$ row vectors of $M_{n-1}$.

\section{Proof of Theorem \ref{theorem:step2}}\label{section:step2}

We first bound the number $N$ of vectors $\Bu$ satisfying the conclusion of Theorem \ref{theorem:step2}. 

Because each GAP is determined by its generators and dimensions, the number of $Q$s is bounded by $(n^{A+O_{B,\ep}(1)})^{O_{B,\ep}(1)} (n^{O_{B,\ep}(1)})^{O_{B,\ep}(1)} = n^{O_{A,B,\ep}(1)}$. 

Next, for a given $Q$ of rank $O_{B,\ep}(1)$ and size $n^{O_{B,\ep}(1)}$, there are at most $n^{n-2n^\ep}|Q|^{n-2n^\ep} = n^{O_{B,\ep}(n)}$ ways to choose the $n-2n^\ep$ components $u_i$ that $Q$ contains.

The remaining components belong to the set $\{p/q, |p|,|q|\le n^{A/2+O_{B,\ep}(1)}\}$, so there are at most $(n^{A+O_{B,\ep}(1)})^{2n^\ep}= n^{O_{A,B,\ep}(n^\ep)}$ ways to choose them. 

Hence, we obtain the key bound 

\begin{equation}\label{eqn:step2:N}
N\le n^{O_{A,B,\ep}(1)}   n^{O_{B,\ep}(n)}  n^{O_{A,B,\ep}(n^\ep)} = n^{O_{B,\ep}(n)}.
\end{equation}

Therefore, for our task of proving Theorem \ref{theorem:step2}, it would be ideal if we can show that the probability $\P_{\beta_0}(\Bu)$ that  $|\langle \Bu,\row_i(M_{n-1}) \rangle | \le  \beta_0$ for $n-O_{B,\ep}(1)$ rows of $M_{n-1}$ are smaller than $\exp(-\Omega(n))/N$  for each $\Bu$, where $\beta_0:=n^{-A/2+O_{B,\ep}(1)}$ is the bound from the conclusion of Theorem \ref{theorem:step1}. 

Roughly speaking, our strategy is to classify $\Bu$ into two classes: one contains of $\Bu$ of very small $\P_{\beta_0}(\Bu)$, and thus their contribution is negligible; the other contains of $\Bu$ of relatively large $\P_{\beta_0}(\Bu)$. To deal with those $\Bu$ of the second type, we will not control $\sum \P_{\beta_0}(\Bu)$ directly but pass to a class of new vectors $\Bu'$ that are also almost orthogonal to many rows of $M_{n-1}$, while the probability $\sum \P_{\beta_0}(\Bu')$ are relatively smaller than $\sum \P_{\beta_0}(\Bu)$. More details follow.    

\subsection{ Technical reductions and key observations}\label{subsection:observation} By paying a factor of $n^{O_{B,\ep}(1)}$ in probability, we may assume that $|\langle \Bu,\row_i(M_{n-1}) \rangle | \le  \beta_0$ for the first  $n-O_{B,\ep}(1)$ rows of $M_{n-1}$ (because $M_{n-1}$ is symmetric, this is the case of most correlations among the entries). Also, by paying another factor of $n^{n^\ep}$ in probability, we may assume that the first $n_0$ components $u_i$ of $\Bu$ belong to a GAP $Q$, and $u_{n_0}\ge 1/2\sqrt{n-1}$, where $n_0:=n-2n^\ep$. We refer the remaining $u_i$ as exceptional components. Note that these extra factors do not affect our final bound $\exp(-\Omega(n))$.

For given $\beta>0$ and $i\le n_0$, we define

$$\rho_{\beta}^{(i)}(\Bu):=\sup_a\P_{x_i,\dots,x_{n_0}}(|x_iu_i+\dots+x_{n_0}u_{n_0}-a|\le \beta),$$

where $x_i,\dots, x_{n_0}$ are iid copies of $\xi$.

A crucial observation is that, by exposing the rows of $M_{n-1}$ one by one, and due to symmetry, the probability $\P_{\beta}(\Bu)$ that $|\langle \Bu,\row_i(M_{n-1})\rangle|\le \beta$ for all $i\le n-O_{B,\ep}(1)$ can be bounded by

\begin{eqnarray}\label{eqn:step2:Pu}
\P_{\beta}(\Bu)&\le& \prod_{1\le i\le n-O_{B,\ep}(1)}\sup_a\P_{x_i,\dots,x_{n-1}}(|x_iu_i+\dots+x_{n-1}u_{n-1}-a|\le \beta)\nonumber \\
&\le&  \prod_{1\le i\le n_0}\sup_a\P_{x_i,\dots,x_{n_0}}(|x_iu_i+\dots+x_{n_0}u_{n_0}-a|\le \beta) \nonumber \\
&=&\prod_{1\le k\le n_0}\rho_{\beta}^{(i)}(\Bu).
\end{eqnarray}

Also, because of Condition \ref{condition:space} and $u_{n_0}\ge 1/2\sqrt{n-1}$, for any $\beta<c_1/2 \sqrt{n-1}$ we have

\begin{eqnarray}\label{eqn:step2:upper}
\rho_{\beta}^{(k)}(\Bu) &\le& \sup_a \P_{x_{n_0}}(|x_{n_0}u_{n_0}-a|\le \beta)\nonumber \\
&\le& 1-c_3,
\end{eqnarray}

and thus, 

$$\P_\beta(\Bu)\le (1-c_3)^{n_0}=(1-c_3)^{(1-o(1))n}.$$

Next, let $C$ be a sufficiently large constant depending on $B$ and $\ep$. We classify $\Bu$ into two classes $\mathcal{B}$ and $\mathcal{B}'$, depending on whether $\P_{\beta_0}(\Bu)\ge n^{-Cn}$ or not. 

Because of \eqref{eqn:step2:N}, and $C$ is large enough, 

\begin{equation}\label{eqn:step2:B'}
\sum_{\Bu\in \mathcal{B}'}\P_{\beta_0}(\Bu)\le n^{O_{B,\ep}(n)}/n^{Cn} \le n^{-n/2}.
\end{equation}

For the rest of the section, we focus on $\Bu \in \mathcal{B}$. Depending on the distribution of the sequences $(\rho_{\beta_j}^{(i)}(\Bu))$, we consider two cases.

\subsection{Approximation for degenerate vectors} Let $\mathcal{B}_1$ be the collection of $\Bu\in \mathcal{B}$ satisfying the following property: for any $n'=n^{1-\ep}$ components $u_{i_1},\dots,u_{i_{n'}}$ among the $u_1,\dots, u_{n_0}$, we have 

\begin{equation}\label{eqn:step2:degenerate}
\sup_a\P_{x_{i_1},\dots,x_{i_{n'}}}(|u_{i_1}x_{i_1}+\dots+u_{i_{n'}}x_{i_{n'}}-a|\le n^{-B-4})\ge (n')^{-1/2+o(1)}.
\end{equation}

For consision we set $\beta=n^{-B-4}$. It follows from Theorem \ref{theorem:ILOlinear:continuous} that, among any $u_{i_1},\dots,u_{i_{n'}}$, there are, say, at least $n'/2+1$ components that belong to an interval of length $2\beta$ (because our GAP now has only one element). A simple argument then implies that there is an interval of length $2\beta$ that contains all but $n'-1$ components $u_i$. (To prove this, arrange the components in increasing order, then all but perhaps the first $n'/2$ and the last $n'/2$ components will belong to an interval of length $2\beta$).  
 
Thus there exists a vector $\Bu'\in (2\beta)\cdot \Z$ satisfying the following conditions.

\begin{itemize}
\item $|u_i-u_i'|\le 2\beta$ for all $i$.
\vskip .2in
\item $u_i'=u$ for at least $n_0-n'$ indices $i$.
\end{itemize}

Because of the approximation, whenever $|\langle \Bu, \row_i(M_{n-1})\rangle|\le \beta_0$, we have 

$$|\langle \Bu',\row_i(M_{n-1})\rangle|\le n^{B+2}(2\beta)+\beta_0=\beta'.$$ 

It is clear, from the bound of $\beta$ and $\beta_0$, that $\beta' \le c_2/2\sqrt{n-1}$, and thus by \eqref{eqn:step2:upper}, 

$$\P_{\beta'}(\Bu') \le (1-c_3)^{(1-o(1))n}.$$

Now we bound the number of $\Bu'$ obtained from the approximation. First, there are $O(n^{n-n_0+n'}) = O(n^{2n^{1-\ep}})$ ways to choose those $u_i'$ that take the same value $u$, and there are just $O(\beta^{-1})$ ways to choose $u$. The remaining components belong to the set $(2\beta)^{-1}\cdot \Z$, and thus there are at most $O((\beta^{-1})^{n-n_0+n'})= O(n^{O_{A,B,\ep}(n^{1-\ep})})$ ways to choose them.

Hence we obtain the total bound

\begin{align*}
\sum_{\Bu'}\P_{\beta'}(\Bu') &\le O(n^{2n^{1-\ep}}) O(n^{O_{A,B,\ep}(n^{1-\ep})}) (1-c_3)^{(1-o(1))n}\\
&\le (1-c_3)^{(1-o(1))n}.
\end{align*}

\subsection{Approximation for non-degenerate vectors} Assume that $\Bu\in \mathcal{B}_2:=\mathcal{B}\backslash \mathcal{B}_1$. 
By exposing the rows of $M_{n-1}$ accordingly, and by paying an extra factor $\binom{n_0}{n'}=O(n^{n^{1-\ep}})$ in probability, we may assume that the components $u_{n_0-n'+1},\dots,u_{n_0}$ satisfy the property 

$$\sup_a\P_{x_{n_0-n'+1},\dots,x_{n_0}}(|u_{n_0-n'+1}x_{n_0-n'+1}+\dots+u_{n_0}x_{n_0}-a|\le n^{-B-4})\le (n')^{-1/2+o(1)}$$
\begin{equation}\label{eqn:step2:non-degenerate}
\le n^{1/2-\ep/2+o(1)}.
\end{equation}

Next, define the following sequence $\beta_k, k\ge 0$,

$$\beta_{k+1}:= (2n^{B+2}+1) \beta_k.$$

Recall from \eqref{eqn:step2:Pu} that 

$$\P_{\beta_k}(\Bu) \le \prod_{1 \le i \le n_0-n'} \rho_{\beta_{k}}^{(i)}(\Bu):=\pi_{\beta_k}(\Bu).$$

Roughly speaking, the reason we truncated the product here is that whenever $i\le n_0-n^{1-\ep}$, and $\beta_k$ is small enough, the terms $\rho_{\beta_k}^{(i)}(\Bu)$ are smaller than $(n')^{-1/2+o(1)}$, owing to \eqref{eqn:step2:non-degenerate}. This fact will allow us to gain some significant factors when applying Theorem \ref{theorem:ILOlinear:continuous}. 

Note that $\pi_{\beta_k}(\Bu)$ increases with $k$, and recall that $\pi_0(\Bu)\ge n^{-Cn}$. Thus, by the pigeon-hole principle, there exists $k_0:=k_0(\Bu)\le C\ep^{-1}$ such that 

\begin{equation}\label{eqn:step2:pigeon-hole}
\pi_{\beta_{k_0+1}}(\Bu) \le n^{\ep n} \pi_{\beta_{k_0}}(\Bu).
\end{equation}

It is crucial to note that, since $A$ was chosen to be sufficiently large compared to $O_{B,\ep}(1)$ and $C$, we have 

$$\beta_{k_0+1}\le n^{-B-4}.$$

Having mentioned the upper bound of $\rho_{\beta_i}^{(i)}(\Bu)$, we now turn to its lower bound. Because of Condition \ref{condition:bound}, and $u_i\le 1$, the following trivial bound holds for any $\beta\ge \beta_0$ and $i\le n_0-n'$,

$$\rho_{\beta}^{(i)}(\Bu) \ge \beta^{-1}n^{B+2} \ge n^{-A/2+O_{B,\ep}(1)}.$$

We next divide the interval $I=[n^{-A/2+O_{B,\ep}(1)}, n^{1/2-\ep/2-o(1)}]$ into $K=(A/2+O_{B,\ep}(1))\ep^{-1}$ sub-intervals $I_k=[n^{-A/2+O_{B,\ep}(1)+ k\ep},n^{-A/2+O_{B,\ep}(1)+(k+1)\ep}]$. For short, we denote by $\rho_k$ the left endpoint of each $I_k$. Thus $\rho_k=n^{-A/2+O_{B,\ep}(1)+k\ep}$.

With all the necessary settings above, we now classify $\Bu$ basing on the distributions of the $\rho_{\beta_{k_0}}^{(i)}(\Bu), 1\le i \le n_0-n^{1-\ep}$ . 

For each $0\le k_0 \le C\ep^{-1}$ and each tuple $(m_0,\dots,m_K)$ satisfying $m_0+\dots+m_K=n_0-n^{1-\ep}$, we let $\mathcal{B}_{k_0}^{(m_0,\dots,m_K)}$ denote the collection of those
$\Bu$ from $\mathcal{B}_2$ that satisfy the following conditions.

\begin{itemize}
\item $k_0(\Bu)= k_0$.

\vskip .2in

\item There are exactly $m_k$ terms of the sequence $(\rho_{\beta_{k_0}}^{(i)}(\Bu))$ belonging to the interval $I_k$. In other words, if $m_0+\dots+m_{k-1}+1 \le i \le m_0+\dots+m_k$ then $\rho_{\beta_{k_0}}^{(i)}(\Bu)\in I_k$.
\end{itemize}

Now we will use Theorem \ref{theorem:ILOlinear:continuous} to approximate $\Bu \in \mathcal{B}_{k_0}^{(m_0,\dots,m_K)} $ as follows. 

\begin{itemize}

\item {\it First step}. Consider each index $i$ in the range $1 \le i\le m_0$. Because $\rho_{\beta_{k_0}}^{(1)}\in I_0$, we apply Theorem \ref{theorem:ILOlinear:continuous} to approximate $u_i$ by $u_i'$ such that $|u_i-u_i'|\le \beta_{i_0}$ and the $u_i'$ belong to a GAP $Q_0$ of rank $O_{B,\ep}(1)$ and size $O(\rho_0^{-1}/n^{1/2-\ep})$ for all but $n^{1-2\ep}$ indices $i$. Furthermore, all $u_i'$ have the form $\beta_{k_0}\cdot p/q$, where $|p|,|q| =O(n \beta_{k_0}^{-1})=O(n^{A/2+O_{B,\ep}(1)})$.  

\vskip .2in

\item {\it $k$-th step, $1\le k\le K$}. We focus on $i$ from the range $n_0+\dots+ n_{k-1}+1\le i\le n_0+\dots+n_k$. Because $\rho_{\beta_{k_0}}^{(n_0+.\dots+n_{k-1}+1)}\in I_k$, we apply Theorem \ref{theorem:ILOlinear:continuous} to approximate $u_i$ by $u_i'$ such that $|u_i-u_i'|\le \beta_{k_0}$ and $u_i$ belongs to a GAP $Q_k$ of rank $O_{B,\ep}(1)$ and size $O(\rho_k^{-1}/n^{1/2-\ep})$ for all but $n^{1-2\ep}$ indices $i$. Furthermore, all $u_i'$ have the form $\beta_{k_0}\cdot p/q$, where $|p|,|q| =O(n \beta_{k_0}^{-1})=O(n^{A/2+O_{B,\ep}(1)})$.  

\vskip .2in

\item For the remaining components $u_i$, we just simply approximate them by the closest point in $\beta_{i_0}\cdot \Z$. 

\end{itemize}

We have thus provided an approximation of $\Bu$ by $\Bu'$ satisfying the following properties. 

\begin{enumerate}
\item $|u_i-u_i'|\le \beta_{k_0}$ for all $i$. 
\vskip .2in
\item $u_i'\in Q_k$ for all but $n^{1-2\ep}$ indices $i$ in the range $m_0+\dots+ m_{k-1}+1\le i\le m_0+\dots+m_k$.
\vskip .2in
\item All the $u_i'$, including the generators of $Q_k$, belong to the set $\beta_{k_0}\cdot \{p/q, |p|,|q|\le n^{A/2+O_{B,\ep}(1)}\}$. 
\vskip .2in
\item $Q_k$ has rank $O_{B,\ep}(1)$ and size $|Q_k|=O(\rho_k^{-1}/n^{1/2-\ep})$.
\end{enumerate}

Let $\mathcal{B'}_{k_0}^{(m_1,\dots,m_K)}$ be the collection of all $\Bu'$ obtained from $\Bu\in \mathcal{B}_{k_0}^{(m_1,\dots,m_K)}$ as above. Observe that, as $|\langle \Bu, \row_i(M_{n-1})\rangle|\le \beta_0$ for all $i\le n-O_{B,\ep}(1)$, we have

\begin{equation}\label{eqn:step2:u'}
|\langle \Bu', \row_i(M_{n-1})\rangle| \le (n^{B+2}+1) \beta_{i_0}.
\end{equation}

Hence, in order to justify Theorem \ref{theorem:step2} in the case $\Bu\in \mathcal{B}_2$, it suffices to show that the probability that \eqref{eqn:step2:u'} holds for all $i\le n-O_{B,\ep}(1)$, for some $\Bu'\in \mathcal{B'}_{k_0}^{(m_1,\dots,m_K)}$, is small. 

Consider a $\Bu'\in \mathcal{B'}_{k_0}^{(m_1,\dots,m_K)}$ and the probability $\P_{\Bu'}$ that  \eqref{eqn:step2:u'} holds for all $i\le n-O_{B,\ep}(1)$. We have

\begin{eqnarray}
\P_{\Bu'}&\le& \prod_{1 \le i\le n_0-n^{1-\ep}} \sup_a \P_{x_i,\dots,x_{n_0}}(|u_i'x_i+\dots+u_{n-1}'x_{n_0}-a|\le (n^{B+2}+1) \beta_{i_0}) \nonumber \\
&\le& \prod_{1 \le i\le n_0-n^{1-\ep}} \sup_a \P_{x_i,\dots,x_{n_0}}(|u_ix_i+\dots+u_{n-1}x_{n_0}-a|\le (2n^{B+2}+1) \beta_{i_0})\nonumber \\
&=& \pi_{\beta_{k_0+1}}(\Bu) \le n^{\ep n}\pi_{\beta_{k_0}}(\Bu), \nonumber
\end{eqnarray}

where in the last inequality we used \eqref{eqn:step2:pigeon-hole}.

We recall from the definition of $\mathcal{B}_{k_0}^{(m_1,\dots,m_K)}$ that 

\begin{align*}
\pi_{\beta_{k_0}}(\Bu) \le \prod_{k=1}^K \rho_{k+1}^{m_k} &=  n^{\ep(m_1+\dots+m_k)} \prod_{k=1}^K \rho_k^{m_k}\\
&\le n^{\ep n} \prod_{k=1}^K \rho_k^{m_k}.
\end{align*}

Hence,

\begin{equation}\label{eqn:step2:Pu'}
\P_{\Bu'} \le n^{2\ep n}  \prod_{k=1}^K \rho_k^{m_k}.
\end{equation}

In the next step we bound the size of $\mathcal{B'}_{k_0}^{(m_1,\dots,m_K)}$.

Because each $Q_k$ is determined by its $O_{B,\ep}(1)$ generators from the set  $\beta_{i_0}\cdot \{p/q, |p|,|q|\le n^{A/2+O_{B,\ep}(1)}\}$, and its dimensions from the integers bounded by $n^{O_{B,\ep}(1)}$, there are $n^{O_{A,B,\ep}(1)}$ ways to choose each $Q_k$. So the total number of ways to choose $Q_1,\dots, Q_K$ is bounded by

$$(n^{O_{A,B,\ep}(1)})^{K}= n^{O_{A,B,\ep}(1)}.$$

Next, after locating $Q_k$, the number $N_1$ of ways to choose $u_i'$ from each $Q_k$ is 

\begin{align*}
N_1&\le \prod_{k=1}^K \binom{m_k}{n^{1-2\ep}} |Q_k|^{m_k-n^{1-2\ep}}\\ 
&\le 2^{m_1+\dots+m_K} \prod_{k=1}^K |Q_k|^{m_k}\\
&\le (O(1))^n \prod_{k=1}^K\rho_k^{-m_k}/n^{(1/2-\ep)(m_1+\dots+m_k)}\\
&\le \prod_{k=1}^K\rho_k^{-m_k}/n^{(1/2-\ep-o(1))n},
\end{align*}

where we used the bound $|Q_k|=O(\rho_k^{-1}/n^{1/2-\ep})$.

The remaining components $u_i'$ can take any value from the set $\beta_{i_0}\cdot \{p/q, |p|,|q|\le n^{A/2+O_{B,\ep}(1)}\}$, so the number $N_2$ of ways to choose them is bounded by

$$N_2 \le (n^{A+O_{B,\ep}(1)})^{2n^\ep + Kn^{1-2\ep}} = n^{O_{A,B,\ep}(n^{1-2\ep})}.$$

Putting the bound for $N_1$ and $N_2$ together, we obtain a bound $N$ for $|\mathcal{B'}_{k_0}^{(m_1,\dots,m_K)}|$,

\begin{equation}\label{eqn:step2:N}
N\le \prod_{k=1}^K\rho_k^{-m_k}/n^{(1/2-\ep-o(1))n}.
\end{equation}

It follows from \eqref{eqn:step2:Pu'} and \eqref{eqn:step2:N} that

\begin{equation}\label{eqn:step2:total}
\sum_{\Bu'\in \mathcal{B'}_{k_0}^{(m_1,\dots,m_K)}}\P_{\Bu'} \le  n^{2\ep n}  \prod_{k=1}^K \rho_k^{m_k}  \prod_{k=1}^K\rho_k^{-m_k}/n^{(1/2-\ep-o(1))n} \le n^{-(1/2-3\ep -o(1))n}.
\end{equation}

Summing over the choices of $k_0$ and $(m_1,\dots,m_K)$ we obtain the bound 

$$ \sum_{k_0,m_1,\dots,m_K} \sum_{\Bu'\in \mathcal{B'}_{k_0}^{(m_1,\dots,m_K)}}\P_{\Bu'} \le n^{-(1/2-3\ep-o(1))n},$$

completing the proof of Theorem \ref{theorem:step2}.

\section{Proof of Corollary \ref{cor:concentration}}\label{section:cor}

Assume that the upper diagonal entries of $M_n$ satisfy the conditions of Corollary \ref{cor:concentration}. We denote by $\lambda_1\le \lambda_2\le \cdots \le \lambda_n$ the real eigenvalues of $M_n$. 

Our first ingredient is the following special form of the spectral concentration result of Guionnet and Zeitouni.  

\begin{lemma}\cite[Theorem 1.1]{GZ}\label{lemma:strongconcentration} 
Assume that $f$ is a convex Lipschitz function. Then for any $\delta \ge \delta_0:=16C \sqrt{\pi}|f|_L/n $,

$$\P\left(|\sum_{i=1}^nf(\lambda_i)-\E(\sum_{i=1}^nf(\lambda_i))|\ge \delta n\right)\le 4\exp(-\frac{n^2(\delta-\delta_0)^2}{16C^2|f|_L^2}).$$ 
\end{lemma}

Following \cite{CV} and \cite{FRZ}, we will apply the above theorem to the cut-off functions $f_\ep^{+}(x):=\log(\max(\ep,x))$ and $f_\ep^{-}(x)=\log(\max(\ep,-x))$, for some $\ep>0$ to be determined. The main reason we have to truncate the $\log$ function is because it is not Lipschitz. Note that $f^{+}$ and $f^{-1}$ both have Lipschitz constant $\ep^{-1}$. Although they are not convex, it is easy to write them as difference of convex functions of Lipschitz constant $O(\ep^{-1})$, and so Lemma \ref{lemma:strongconcentration} applies. Thus the followings hold for $\delta\gg (\ep n)^{-1}$

$$\P\left(|\sum_{\lambda_i\in S_\ep^{+}}\log\lambda_i-\E(\sum_{\lambda_i\in S_\ep^{+}}\log\lambda_i)|\ge \delta n\right)\le \exp(-\Theta(n^2\delta^2\ep^2))$$

and

$$\P\left(|\sum_{\lambda_i\in S_\ep^{-}}\log|\lambda_i|-\E(\sum_{\lambda_i\in S_\ep^{+}}\log|\lambda_i|)|\ge \delta n\right)\le \exp(-\Theta(n^2\delta^2\ep^2)),$$

where $S_\ep^{+}:=\{\lambda_i, \lambda_i\ge \ep\}$ and $S_\ep^{-}:=\{\lambda_i, \lambda_i\le -\ep\}$.  

Hence,

\begin{equation}\label{eqn:common}
\P\left(|\sum_{\lambda_i\in S_\ep^{-}\cup S_\ep^{+}}\log|\lambda_i|-\E(\sum_{\lambda_i\in S_\ep^{-}\cup S_\ep^{+}}\log|\lambda_i|)|\ge 2\delta n\right)\le \exp(-\Theta(n^2\delta^2\ep^2)).
\end{equation}

Roughly speaking, \eqref{eqn:common} implies that $\prod_{\lambda_i\in S_\ep^{-}\cup S_\ep^{+}}|\lambda_i|$ is well concentrated around its mean. It thus remains to control the factor $R:=\prod_{|\lambda_i|\le \ep}|\lambda_i|$. We will bound $R$ away from zero, relying on Theorem \ref{theorem:singularvalue} and Lemma \ref{lemma:weakconcentration} below.

\begin{lemma}\cite[Proposition 66]{TVlocal}, \cite[Theorem 5.1]{ESY} \label{lemma:weakconcentration}  Assume that $M_n$ is a random symmetric matrix of entries satisfying the conditions of Corollary \ref{cor:concentration}. Then for all $I\subset \R$ with $|I|\ge K^2\log^2n/n^{1/2}$, one has 
$$N_I\ll n^{1/2}|I|$$
with probability $1-\exp(\omega(\log n))$, where $N_I$ is the number of $\lambda_i$ belonging to $I$.
\end{lemma}

By Lemma \ref{lemma:weakconcentration}, we have $|\{i, |\lambda_i|\le \ep \}| \ll n^{1/2} \ep$. Also, Theorem \ref{theorem:singularvalue} implies that $\min_i\{|\lambda_i|\}\ge n^{-A}$ with probability $1-O(n^{-B})$. Thus 

\begin{equation}\label{eqn:prop:1}
R=\prod_{|\lambda_i|\le \ep} |\lambda_i| \ge (\min_i\{|\lambda_i|\})^{n^{1/2}\ep}=n^{-O(n^{1/2}\ep)}.
\end{equation}

Our next goal is the following result.

\begin{proposition}\label{prop:2}
With probability $1-n^{\omega(1)}$ we have

\begin{equation}\label{eqn:concentration1}
\prod_{\lambda_i\in S_\ep^{-}\cup S_{\ep}^{+}}|\lambda_i| =\exp(-O(\ep^{-2}))\E(\prod_{\lambda_i\in S_\ep^{-}\cup S_\ep^{+}}|\lambda_i|)-\exp(\frac{2\log n}{\ep}).
\end{equation}
\end{proposition}

Let us complete the proof of Corollary \ref{cor:concentration} assuming Proposition \ref{prop:2}.

Firstly, because $\prod_{\lambda_i\in S_\ep^{-}\cup S_\ep^{+}}|\lambda_i| \ge \prod_{i=1}^n|\lambda_i|/\ep^{n - |S_\ep^{-}\cup S_\ep^{+}|} \ge \prod_{i=1}^n|\lambda_i| =|\det(M_n)|$, it follows from Proposition \ref{prop:2} that with probability $1-n^{\omega(1)}$,

\begin{equation}\label{eqn:mainterm}
\prod_{\lambda_i\in  S_\ep^{-}\cup S_\ep^{+}}|\lambda_i| =\exp(-O(\ep^{-2}))\E(|\det(M_n)|)-\exp(\frac{2\log n}{\ep}).
\end{equation}

Secondly, by \eqref{eqn:prop:1}, the following holds with probability $1-O(n^{-B})$

$$|\det(M_n)|=\prod_{\lambda_i \notin  S_\ep^{-}\cup S_\ep^{+}}|\lambda_i| \prod_{\lambda_i\in  S_\ep^{-}\cup S_\ep^{+}}|\lambda_i| \ge n^{-O(n^{1/2}\ep)}\prod_{\lambda_i\in  S_\ep^{-}\cup S_\ep^{+}}|\lambda_i|.$$

Combining with \eqref{eqn:mainterm}, we have 

$$|\det(M_n)|=\exp(-O(\ep^{-2}+\ep n^{1/2}\log n))\E(|\det(M_n|)-n^{-O(n^{1/2}\ep)}\exp(\frac{2\log n}{\ep}).$$

By choosing $\ep=n^{-1/6}$, we obtain the conclusion of Corollary \ref{cor:concentration}, noting that $\E(|\det(M_n)|)\gg \exp(n)$.

It remains to prove Proposition \ref{prop:2}.
\begin{proof}(of Proposition \ref{prop:2}) Set $\delta:=\frac{\log n}{\ep n}$, which satisfies the condition of Lemma \ref{lemma:strongconcentration} as $n$ is sufficiently large. By \eqref{eqn:common} we have

\begin{eqnarray}\label{eqn:proof:2}
\P(|\sum_{\lambda_i\in S_\ep^{-}\cup S_\ep^{+}}\log |\lambda_i| - \E(\sum_{\lambda_i\in S_\ep^{-} \cup S_\ep^{+}}\log |\lambda_i|)| \ge \frac{2\log n}{\ep}) \le \exp(-\Theta(\log^2 n))\nonumber\\
\P(|\prod_{\lambda_i\in S_\ep^{-}\cup S_\ep^{+}}|\lambda_i| - \exp(\E(\sum_{\lambda_i\in S_\ep^{-}\cup S_\ep^{+}}\log |\lambda_i|)|) \ge \exp(\frac{2\log n}{\ep})) \le \exp(-\Theta(\log^2 n)).
\end{eqnarray}

Set 

$$U:= \sum_{\lambda_i\in S_\ep^{-}\cup S_\ep^{+}}\log |\lambda_i| - \E(\sum_{\lambda_i\in S_\ep^{-}\cup S_\ep^{+}}\log |\lambda_i|).$$ 

We have $\E(U)=0$. Thus, by Jensen inequality and by \eqref{eqn:common},

\begin{equation}\label{eqn:proof:3}
1\le \E(\exp(U))\le \E(\exp(|U|))\le 1+ \int_{t} \exp(t) \P(|U|\ge t) =\exp(O(\ep^{-2})).
\end{equation}

Observe that 

$$\E(\exp(U)) = \E(\prod_{\lambda_i\in S_\ep^{-}\cup S_\ep{+}}|\lambda_i|)/\exp(\E(\sum_{\lambda_i\in S_\ep^{-}\cup S_\ep^{+}}\log |\lambda_i|)).$$ 

It thus follows from \eqref{eqn:proof:3} that 

$$\exp(\E(\sum_{\lambda_i\in S_\ep^{-}\cup S_\ep^{+}}\log |\lambda_i|))= \exp(-O(\ep^{-2}))\E(\prod_{\lambda_i\in S_\ep^{-}\cup S_\ep{+}}|\lambda_i|).$$

This relation, together with  \eqref{eqn:proof:2}, imply that with probability $1-n^{\omega(1)}$,

$$\prod_{\lambda_i\in S_\ep^{-}\cup S_\ep^{+}}|\lambda_i| =\exp(-O(\ep^{-2}))\E(\prod_{\lambda_i\in S_\ep^{-}\cup S_\ep^{+}}|\lambda_i|)-\exp(\frac{2\log n}{\ep}). $$

\end{proof}

\appendix

\section{Proof of Lemma \ref{lemma:Odlyzko}}\label{appendix:Odlyzko}
Assume that $\Bv_1,\dots,\Bv_k \in \R^n$ are independent vectors that span $H$. Also, without loss of generality, we assume that the subvectors $(v_{11},\dots,v_{1k}),\dots,(v_{k1},\dots,v_{kk})$ generate a full space of dimension $k$. 

Consider a random vector $\Bu=(f_1+x_1,\dots,f_n+x_n)$, where $x_1,\dots,x_n$ are iid copies of $\xi$. If $\Bu\in H$, then there exist $\alpha_1,\dots,\alpha_k$ such that 

$$\Bu=\sum_{i=1}^k \alpha_i \Bv_i.$$ 

Note that $\alpha_1,\dots,\alpha_k$ are uniquely determined once the first $k$ components of $\Bu$ are exposed. Thus we have

$$\P(\Bu\in H)\le \prod_{k+1\le j}\P_{x_j}(x_j+f_j=\sum_{i=1}^k \alpha_iv_{ij})\le (\sqrt{1-c_3})^{n-k},$$

where in the last estimate we use the fact (which follows from Condition \ref{condition:space}) that $\sup_a\P(\xi=a) \le \sqrt{1-c_3}$.

\section{Proof of Lemma \ref{lemma:decoupling}}\label{appendix:decoupling}
The goal of this section is to establish the inequality

$$\P_{\Bv,\Bw}(|\sum_{i,j} A_U(ij)v_iw_j|=O_{B,\ep}(\beta \sqrt{\log n})) \ge \rho^8/((2\pi)^{7/2}\exp(8\pi)),$$

under the assumption that

$$\sup_a \P_{\Bx} (|\sum_{i,j}a_{ij}(x_i+f_i)(x_j+f_j)-a|\le \beta)=\rho \ge n^{-B}.$$

Set 

$$a_{ij}':=a_{ij}/\beta.$$

We have

\begin{align*}
\sup_{a'} \P_{\Bx} (|\sum_{i,j}a_{ij}'(x_i+f_i)(x_j+f_j)-a'|\le 1)&=\sup_{a} \P_{\Bx} (|\sum_{i,j}a_{ij}(x_i+f_i)(x_j+f_j)-a|\le 1)\\
&\ge n^{-B}.
\end{align*}

We next write  

\begin{align*}
\P_\Bx(|\sum_{i,j}a_{ij}'(x_i+f_i)(x_j+f_j)-a'|\le 1) &= \P\Big(\exp(-\frac{\pi}{2}(\sum_{i,j}a_{ij}'(x_i+f_i)(x_j+f_j)-a')^2 \ge \exp(-\frac{\pi}{2} )\Big)\\
& \le \exp(\frac{\pi}{2} ) \E_\Bx \exp(-\frac{\pi}{2}(\sum_{i,j}a_{ij}'(x_i+f_i)(x_j+f_j)-a')^2).
\end{align*} 

Note that $\exp(-\frac{\pi}{2} x^2) =\int_{\R} e(xt) \exp(-\frac{\pi}{2} t^2) dt$. Thus 

$$\P_\Bx(|\sum_{i,j}a_{ij}'(x_i+f_i)(x_j+f_j)-a'|\le 1) \le \exp(\frac{\pi}{2}) \int_{\R} |\E_\Bx e[(\sum_{i,j}a_{ij}'(x_i+f_i)(x_j+f_j))t]| \exp(-\frac{\pi}{2} t^2)dt$$
$$\le \exp(\frac{\pi}{2})\sqrt{2\pi} \int_{\R} |\E_\Bx e[(\sum_{i,j}a_{ij}'(x_i+f_i)(x_j+f_j))t]| (\exp(-\frac{\pi}{2} t^2)/\sqrt{2\pi})dt.$$

Consider $\Bx$ as $(\Bx_U,\Bx_{\bar{U}})$, where $\Bx_U, \Bx_{\bar{U}}$ are the vectors corresponding to $i\in U$ and  $i\notin U$ respectively. By the Cauchy-Schwarz inequality we have

\begin{align*}  
&\quad\Big[\int_{\R} \big|\E_\Bx e((\sum_{i,j}a_{ij}'(x_i+f_i)(x_j+f_j))t) \big|(\exp(-\frac{\pi}{2} t^2)/\sqrt{2\pi}) dt\Big]^4 \\
&\le \Big[\int_{\R}  \big|\E_\Bx e((\sum_{i,j}a_{ij}'(x_i+f_i)(x_j+f_j))t)\big|^2 (\exp(-\frac{\pi}{2} t^2)/\sqrt{2\pi}) dt \Big]^2\\
&\le \Big[\int_{\R} \E_{\Bx_U}\big|\E_{\Bx_{\bar{U}}}e((\sum_{i,j} a_{ij}'(x_i+f_i)(x_j+f_j))t)\big|^2 (\exp(-\frac{\pi}{2} t^2)/\sqrt{2\pi}) dt\Big]^2 \\ 
&=\Big[ \int_{\R} \E_{\Bx_U}\E_{\Bx_{\bar{U}},\Bx_{\bar{U}}'} e\Big(\big(\sum_{i\in U,j\in \bar{U}}a_{ij}'(x_i+f_i)(x_j-x_j')+ \\
&+\sum_{i\in \bar{U},j\in \bar{U}}a_{ij}'((x_i+f_i)(x_j+f_j)-(x_i'+f_i)(x_j'+f_j))\big)t\Big) (\exp(-\frac{\pi}{2} t^2)/\sqrt{2\pi})  dt\Big]^2\\
&\le \int_{\R} \E_{\Bx_{\bar{U}},\Bx_{\bar{U}}'}\Big|\E_{\Bx_{U}}e\Big(\big(\sum_{i\in U,j\in \bar{U}}a_{ij}'(x_i+f_i)(x_j-x_j')+\\
&+\sum_{i\in \bar{U},j\in \bar{U}}a_{ij}'((x_i+f_i)(x_j+f_j)-(x_i'+f_i)(x_j'+f_j))\big)t\Big)\Big|^2 (\exp(-\frac{\pi}{2} t^2)/\sqrt{2\pi}) dt\\
&=\int_{\R} \E_{\Bx_U,\Bx_U',\Bx_{\bar{U}},\Bx_{\bar{U}}'} e(\big(\sum_{i\in U, j\in \bar{U}}a_{ij}'(x_i-x_i')(x_j-x_j')\big)t\Big) (\exp(-\frac{\pi}{2} t^2)/\sqrt{2\pi}) dt\\
&=\int_{\R}\E_{\By_{U},\Bz_{\bar{U}}}e\Big((\sum_{i\in \bar{U},j\in U} a_{ij}'y_iz_j)t\Big)(\exp(-\frac{\pi}{2} t^2)/\sqrt{2\pi}) dt,
\end{align*}

where $\By_{U}=\Bx_{U}-\Bx_{U}'$ and $\Bz_{\bar{U}}=\Bx_{\bar{U}}-\Bx_{\bar{U}}'$, whose entries are iid copies of $\xi-\xi'$.

Thus we have 

$$\Big[\int_{\R} \big|\E_\Bx e((\sum_{i,j}a_{ij}'x_ix_j)t)\big|(\exp(-\frac{\pi}{2} t^2)/\sqrt{2\pi}) dt\Big]^8\le \Big[\int_{\R}\E_{\By_U,\Bz_{\bar{U}}}e\big((\sum_{i\in U,j\in \bar{U}} a_{ij}'y_iz_j)t\big) (\exp(-\frac{\pi}{2} t^2)/\sqrt{2\pi}) dt\Big]^2$$

$$\le \int_{\R}\E_{\By_U,\Bz_{\bar{U}}, \By_U', \Bz_{\bar{U}}'}e\big((\sum_{i\in U,j\in \bar{U}}a_{ij}'y_iz_j-\sum_{i\in U, j\in \bar{U}} a_{ij}'y_i' z_j')t\big)(\exp(-\frac{\pi}{2} t^2)/\sqrt{2\pi}) dt.$$

Because $a_{ij}'=a_{ji}'$, we can write the last term as

\begin{align*}
&\int_{\R}\E_{\By_U,\Bz_{\bar{U}}', \By_U', \Bz_{\bar{U}}}e\Big(\big(\sum_{i\in U,j\in \bar{U}}a_{ij}'y_iz_j+\sum_{j\in \bar{U}, i\in U} a_{ji}(-z_j')y_i'\big)t\Big)(\exp(-\frac{\pi}{2} t^2)/\sqrt{2\pi}) dt\\
&=\int_{\R} \E_{\Bv,\Bw}e\big(( \sum_{i\in U, j\in \bar{U}}a_{ij}'v_iw_j + \sum_{i\in \bar{U}, j\in U} a_{ij}'v_iw_j)t\big)(\exp(-\frac{\pi}{2} t^2)/\sqrt{2\pi}) dt,
\end{align*}

where $\Bv:=(\By_U,-\Bz_{\bar{U}}')$ and $\Bw:=(\By_U', \Bz_{\bar{U}})$. 

Next, we observe that   

$$\int_{\R} \E_{\Bv,\Bw}e\Big(\big( \sum_{i\in U j\in \bar{U}}a_{ij}'v_iw_j + \sum_{i\in \bar{U}, j\in U} a_{ij}'v_iw_j\big)t\Big)(\exp(-\frac{\pi}{2} t^2)/\sqrt{2\pi}) dt$$ 
$$= (1/\sqrt{2\pi})\E_{\Bv,\Bw}\exp(-\frac{\pi}{2}(\sum_{i,j}A_U(ij)'v_iw_j)^2),$$

where 

$$A_U(ij)'=A_U(ij)/\beta.$$

Thus 

\begin{align*}
\rho^8 = \Big(\P_\Bx(|\sum_{i,j}a_{ij}'x_i,x_j-a'|\le 1)\Big)^8 &\le \exp(4\pi)(2\pi)^4 (\int_{\R} |\E_\Bx e((\sum_{i,j}a_{ij}'x_ix_j)t)| (\exp(-\frac{\pi}{2} t^2)/\sqrt{2\pi}) dt)^8\\
&\le \exp(4\pi)(2\pi)^{7/2} \E_{\Bv,\Bw}\exp(-\frac{\pi}{2}(\sum_{i,j}A_U(ij)'v_iw_j)^2).
\end{align*}

Because $\rho\ge n^{-B}$, the inequality above implies that 

$$\P_{\Bv,\Bw}(|\sum_{i,j} A_U(ij)'v_iw_j|=O_{B,\ep}(\sqrt{\log n})) \ge \rho^8/((2\pi)^{7/2}\exp(4\pi)).$$

Scaling back to $A_{ij}$, we obtain

$$\P_{\Bv,\Bw}(|\sum_{i,j} A_U(ij)v_iw_j|=O_{B,\ep}(\beta \sqrt{\log n})) \ge \rho^8/((2\pi)^{7/2}\exp(4\pi)),$$

completing the proof.

{\bf Acknowledgements.} The author would like to thank R. Pemantle and V. Vu for very helpful discussions and comments.

\end{document}